\documentclass[11pt]{article}
\usepackage{amsmath,amsthm,amssymb}
\usepackage[T1]{fontenc}
\usepackage[utf8]{inputenc}
\usepackage[dvipdf]{graphicx}
\usepackage{color}
\usepackage{epstopdf}
\usepackage{dsfont}
\usepackage{mathtools}
\usepackage{enumerate}
\usepackage{xcolor}
\usepackage{mathrsfs}
\usepackage[normalem]{ulem}
\usepackage[colorlinks=true,citecolor=red,linkcolor=db,urlcolor=blue,pdfstartview=FitH]{hyperref}
\usepackage[numbers]{natbib}

\allowdisplaybreaks[4]

\providecommand{\keywords}[1]
{
  \small	
  \textbf{\textit{Keywords:}} #1
}

\providecommand{\amsclassification}[1]
{
  \par\smallskip
  \small
  \textbf{\textit{2020 Mathematics Subject Classification:}} #1
}

\definecolor{db}{RGB}{0, 0, 130}
\definecolor{rp}{rgb}{0.25, 0, 0.75}
\definecolor{dg}{rgb}{0, 0.6, 0}

\newtheorem{theorem}{Theorem}[section]
\newtheorem*{acknowledgement}{Acknowledgement}
\newtheorem{definition}{Definition}[section]
\newtheorem{prop}[definition]{Proposition}

\newtheorem{assumption}[definition]{Assumption}
\newtheorem{lemma}[definition]{Lemma}

\newtheorem{proposition}[definition]{Proposition}
\newtheorem{remark}[definition]{Remark}

\numberwithin{equation}{section}

\def\pa{\partial}
\def\d{\delta}

\def\a{\alpha}
\def\e{\epsilon}

\def\Ac{\mathcal A}
\def\Jc{\mathcal J}
\def\Bc{\mathcal B}

\def\Fc{\mathcal F}

\def\Hc{\mathcal H}
\def\Lc{\mathcal L}
\def\Mc{\mathcal M}

\def\Pc{\mathcal P}

\def\Wc{\mathcal W}

\def\ut{\tilde{u}}
\def\vt{\tilde{v}}

\def\bt{\tilde{b}}

\def\E{\mathbb{E}}
\def\F{\mathbb{F}}

\def\N{\mathbb N}

\def\R{\mathbb R}

\def\vb{\bar{v}}

\def\ub{\bar{u}}
\def\Jb{\bar{J}}
\def\mub{\bar{\mu}}
\def\mut{\tilde{\mu}}

\def\eps{\epsilon}

\def\x{\times}
\def\Om{\Omega}

\newcommand{\norm}[1]{\left\lVert #1\right\rVert}
\newcommand{\abs}[1]{\left\lvert #1\right\rvert}

\def \endproof{\hbox{ }\hfill$\Box$}

\title{A comparison principle for Wasserstein PDEs with state- and law-dependent common noise}
\author{Erhan Bayraktar
        \footnote{Department of Mathematics, University of Michigan. erhan@umich.edu. E. Bayraktar is partially supported by the NSF Grants DMS-2507940 and DMS-2406232 and by the Susan M. Smith Professorship.}
        \and Ibrahim Ekren
        \footnote{Department of Mathematics, University of Michigan. iekren@umich.edu. I. Ekren is partially supported by the NSF Grant DMS-2406240.}
        \and Xihao He
        \footnote{Department of Mathematics, University of Southern California. 
        xihaohe@usc.edu}
        \and Xin Zhang
        \footnote{Department of Finance and Risk Engineering, New York University. xz1662@nyu.edu. X. Zhang is partially supported by the NSF Grant DMS-2508556.}
        }
\date{\today}

\begin{document}
\maketitle

\begin{abstract}
We prove a comparison principle for a class of second-order Hamilton--Jacobi--Bellman equations on the Wasserstein space whose second-order term is generated by a general common-noise Hessian. The main difficulty is that the relevant second-order direction is induced by a state- and measure-dependent coefficient, so the associated perturbation of the measure is no longer a translation or a fixed state-dependent transformation. We introduce a nonlinear flow of measures
\[
    \partial_m\psi(x,\mu,m)
    =
    \sigma_0(\psi(x,\mu,m),\psi(\cdot,\mu,m)_\#\mu),
    \qquad
    \psi(x,\mu,0)=x,
\]
and use it to transform the Wasserstein-space equation into an augmented equation on $[0,T]\times \mathcal P_2(\mathbb R)\times\mathbb R$, where the general Hessian becomes an ordinary second derivative in the auxiliary variable. The construction may be viewed as a measure-dependent Lamperti transform: it removes the common-noise direction at the level of the equation, but unlike the classical one-dimensional Lamperti transform it permits degeneracy of the coefficient and dependence on the conditional law. We establish the spatial, measure-derivative, and negative-Sobolev estimates for this flow that are needed in the viscosity argument. Under structural assumptions on the transformed Hamiltonian, these estimates yield a Crandall--Ishii type comparison theorem for semicontinuous viscosity sub- and supersolutions. This gives, to the best of our knowledge, the first viscosity comparison framework of this kind for the filtering-driven equations considered here, and opens a new class of second-order PDEs on spaces of measures with state- and law-dependent common-noise directions. As an application, we identify the value function of a controlled stochastic filtering problem with state- and law-dependent common noise as the unique viscosity solution of its dynamic programming equation. We also explain how the same change-of-variable viewpoint applies to Zakai-type Kolmogorov equations on spaces of finite positive measures.
\end{abstract}

\keywords{Viscosity solutions; Wasserstein space; Hamilton--Jacobi--Bellman equations; comparison principle; common noise; stochastic filtering; mean field control.}

\amsclassification{Primary 35D40, 35R15; Secondary 49L25, 49N80, 60H30, 93E11.}

\tableofcontents
\section{Introduction}\label{sec:introduction}
Hamilton--Jacobi equations on spaces of probability measures appear naturally in stochastic control, mean field games, and partially observed control problems where the state variable of the dynamic programming equation is a law. The analytic theory combines ideas from viscosity solutions, optimal transport, and stochastic control. Derivatives with respect to measures in the sense of Lions provide a differential calculus on Wasserstein space; see the lectures of Lions~\cite{Lions} and the monographs~\cite{carmona_probabilistic_2018_vol1,carmona_probabilistic_2018_vol2,lions_annals,bensoussan_mean_2013}. Dynamic programming equations for McKean--Vlasov control and related mean field control problems were developed in, among others,~\cite{MR3630288,pham_bellman_2018,cosso2021master,BurzoniIgnazioReppenSoner2020}. On the master-equation side, Gangbo and M{\'e}sz{\'a}ros~\cite{GangboMeszaros2022} and Gangbo, M{\'e}sz{\'a}ros, Mou, and Zhang~\cite{GangboMeszarosMouZhang2022} developed global well-posedness theories based on displacement convexity and displacement monotonicity. A foundational second-order generator on Wasserstein space is the partial Laplacian of Chow and Gangbo~\cite{ChowGangbo2019}. Second-order PDEs on Wasserstein space also arise in finite-dimensional and particle approximations, master-equation methods for propagation of chaos, and learning problems under partial monitoring; see~\cite{GangboMayorgaSwiech2021,Talbi2024FiniteDimensional,BayraktarEkrenZhang2025Particle,CecchinDaudinJacksonMartini2024,BayraktarEkrenZhou2025CBO,BayraktarEkrenZhang2023Regret}. These equations are now a central object in the PDE approach to mean field control and games, but their viscosity theory is still considerably less mature than the finite-dimensional theory because the natural state space has neither local compactness nor a linear differentiable structure.

The presence of common noise leads to genuinely second-order equations on the space of probability measures. From the PDE point of view, the challenge is that the second-order term acts in directions generated by random perturbations of the conditional law. Classical viscosity methods in finite and infinite dimensions go back to~\cite{usersguide,lions1988viscosity1,lions1989viscosity2,lions1989viscosity3,soner1988hamilton}. Finite-dimensional approximation methods and intrinsic viscosity notions for Hamilton--Jacobi--Bellman equations on spaces of probability measures were developed in~\cite{GangboMayorgaSwiech2021,Talbi2024FiniteDimensional}. Recent comparison arguments based on Wasserstein-distance penalizations and smooth approximations include~\cite{Bertucci2025StochasticOT,BertucciLions2024Approx}. The smooth variational principle on Wasserstein space in~\cite{BayraktarEkrenZhang2023Variational} provides a related tool for optimization and viscosity arguments in the absence of local compactness. In the first-order mean field control setting, Soner and Yan~\cite{SonerYan2024Torus,SonerYan2024Eikonal} prove viscosity comparison and uniqueness results on Wasserstein space using Fourier and Sobolev representations of measure distances. In Wasserstein space, comparison principles for second-order and common-noise equations have recently been obtained in several settings; see, for example,~\cite{BaEkZh23,BEHZ,BCEQTZ25,cheung2023viscosity,DaJaSe23,SaBe24,zhou2024viscosity}. The works~\cite{BaEkZh23,DaJaSe23,SaBe24} treat important semilinear or structurally restricted Hamiltonians; \cite{BCEQTZ25} develops a fully second-order Crandall--Lions framework for mean field control with common noise; and \cite{BEHZ} proves comparison for semicontinuous solutions directly on the Wasserstein space. These works show that a successful comparison proof must combine a suitable penalization of probability measures, an Ishii-type lemma, and estimates that compensate for the weak compactness and lack of linear structure of the state space. They do not, however, cover the filtering-type equations treated in the present paper, where the common-noise direction itself is generated by a state- and measure-dependent coefficient and, in the Zakai formulation, the natural state space is a space of finite positive measures.

This paper studies a class of second-order equations whose Hessian direction is generated by a coefficient $\sigma_0(x,\mu)$ depending simultaneously on the state and on the law. This dependence is the main new feature, and it leads to a class of Wasserstein and finite-measure PDEs that, to the best of our knowledge, has not previously been treated by a Crandall--Ishii viscosity definition based on semijets and comparison. If $\sigma_0$ is constant, or if the induced transformation of the measure is fixed, the common-noise Hessian can often be reduced to a more standard second derivative. When $\sigma_0$ depends on the evolving conditional law, however, the perturbation of the measure is governed by a nonlinear transport equation coupled with its own pushforward. The relevant flow admits the equivalent formulations
\begin{align}
    \partial_m\psi(x,\mu,m)
    &=
    \partial_x\psi(x,\mu,m)\sigma_0(x,\mu)
    +
    \int_\mathbb R D_\mu\psi(x,\mu,m)(y)\sigma_0(y,\mu)\mu(dy),
    \qquad
    \psi(x,\mu,0)=x, \label{eq:intro-lions-flow}\\
    \partial_m\psi(x,\mu,m)
    &=
    \sigma_0(\psi(x,\mu,m),\mu_m),
    \qquad
    \mu_m=\psi(\cdot,\mu,m)_\#\mu,
    \qquad
    \psi(x,\mu,0)=x. \label{eq:intro-mkv-flow}
\end{align}
The first formulation identifies the Lions-derivative direction defining the Hessian, while the second gives a McKean--Vlasov characteristic representation. A key preliminary step is to prove that these formulations are equivalent and to obtain estimates for the forward and backward flows. This is not merely a technical rewriting. It is the measure-valued analogue of the classical Lamperti transform: one introduces a coordinate adapted to the noise coefficient and then writes the equation in that adapted coordinate. The distinction is that the coordinate itself depends on the current law and therefore must be estimated simultaneously in the state and measure variables.

Our first main contribution is a regularity theory for this measure-dependent change of variables. Under suitable smoothness assumptions on $\sigma_0$ and its flat derivative, we establish bounds for the spatial derivatives of the backward flow, its measure derivatives, and the stability of the induced pushforward in negative Sobolev norms. These estimates also imply Sobolev composition bounds for functions composed with the flow. They are essential because the comparison proof later tests the Hamiltonian against Fourier-Wasserstein penalizations, whose derivatives must remain controlled after the nonlinear change of variables.

Our second, and central, contribution is the Crandall--Ishii viscosity comparison principle for the transformed PDE. Given a candidate solution $u(t,\mu)$, define
\[
    \bar u(t,\mu,m)
    :=
    u(t,\psi(\cdot,\mu,m)_\#\mu).
\]
The flow is chosen so that the general Wasserstein Hessian $\mathcal H u$ is represented as $\partial_{mm}^2\bar u$ in the augmented variable. The original equation on $\mathcal P_2(\mathbb R)$ is thus transformed into an equation on $[0,T]\times\mathcal P_2(\mathbb R)\times\mathbb R$. We formulate viscosity sub- and supersolutions through the usual Crandall--Ishii semijet machinery on this augmented state space and prove comparison under structural assumptions on the Hamiltonian. This point is important: the definition is not a classical-solution or density-based formulation, and it does not rely on lifting the problem to a Hilbert space where comparison is already available. The proof follows the doubling-of-variables strategy and uses the flow estimates mentioned above. In this sense the transform plays two roles: probabilistically, it removes the common-noise direction from the conditional dynamics, and analytically, it converts a nonstandard Wasserstein Hessian into an ordinary second derivative for which an Ishii argument can be applied.

The comparison theorem applies to controlled stochastic filtering with common noise. In the separated formulation, the controlled state is the conditional law of the signal. We allow the common-noise coefficient to depend on both the signal state and this conditional law. Our main result in this direction is Theorem~\ref{thm:viscosity_property}: under the stated regularity and structural assumptions, the value function of the partially observed control problem is continuous and is the unique Crandall--Ishii viscosity solution of the associated Wasserstein HJB equation. The proof relies on a commutator estimate in negative Sobolev spaces. The closest dynamic programming predecessor for partial observation is the work of Bandini, Cosso, Fuhrman, and Pham~\cite{MR3907014}, which develops randomized filtering and a Bellman equation in Wasserstein space for partially observed control. The earlier infinite-dimensional HJB theory of Gozzi and {\'S}wi{\k e}ch~\cite{GozziSwiech2000} treats optimal control of the Duncan--Mortensen--Zakai equation in weighted Hilbert and Sobolev spaces and is especially relevant for the commutator estimates used later. Related randomized dynamic programming methods for McKean--Vlasov dynamics appear in~\cite{bayraktar2018randomized}; for background on stochastic filtering see~\cite{MR2454694}. Our contribution is different in two ways: we prove a Crandall--Ishii comparison theorem for semicontinuous viscosity solutions, and we allow the common-noise direction in the filtering equation to be generated by a state- and law-dependent coefficient. There is also a closely related line of work by Martini on Kolmogorov equations associated with nonlinear filtering. In~\cite{Martini2023}, Martini studies backward Kolmogorov equations on spaces of finite positive measures and probability measures associated with the Zakai and Kushner--Stratonovich equations, proving classical well-posedness without passing through densities. In~\cite{Martini2024}, he develops a viscosity approach for the Kolmogorov equation on the space of probability measures associated with the nonlinear filtering equation. Although those works do not study control problems, the PDEs they identify contain precisely the measure-valued filtering structures that motivate our change-of-variable method. To the best of our knowledge, no previous work gives a Crandall--Ishii comparison theory for this filtering class of PDEs, and the existing Wasserstein HJB comparison results do not cover these Zakai-type equations in the semicontinuous viscosity framework. The finite-positive-measure transformation discussed below is designed to address this gap. Compared with these works and with the second-order Wasserstein comparison results in~\cite{BEHZ,BCEQTZ25}, the present framework accommodates a common-noise coefficient of the form $\sigma_0(x,\mu)$ and thereby covers a substantially broader family of measure-valued second-order equations.

The rest of the paper is organized as follows. Subsections~\ref{subsec:notation} and~\ref{subsec:measure-derivatives} collect the notation and measure-derivative conventions used throughout the paper. Section~\ref{sec:change} introduces the common-noise Hessian, constructs the nonlinear transform, and formulates the augmented viscosity problem. Section~\ref{sec:comparison} proves the Ishii lemma and the comparison theorem. Section~\ref{sec:applications} applies the result to controlled stochastic filtering and to a related mean field control problem. Section~\ref{sec:lamperti} explains the Lamperti structure of the transform and its probabilistic meaning. Section~\ref{sec:finite-measure} describes the corresponding finite-measure equations arising from nonlinear filtering, including the Zakai-type equations studied by Martini. Sections~\ref{sec:commutator} and~\ref{sec:flow-estimates} record the Sobolev commutator and flow estimates used in the comparison argument.

    \subsection{Notation}\label{subsec:notation}
    Let $T$ be the finite horizon with $0 < T < +\infty$.
    For any measurable space $(\Om,\Fc)$,
    let $\Pc(\Om)$ denote the collection of all probability measures on $(\Om,\Fc)$.
    
    For any Polish space $(E,d)$,
    let $C([0,T],E)$ denote the space of all $E$-valued continuous functions on $[0,T]$ equipped with the uniform norm $\|\cdot\|$,
    $C_b(E)$ denote the collection of all $\R$-valued bounded continuous functions on $E$,
    $\Pc_2(E)$ denote the space of all probability measures $\rho$ on $E$ such that
    \begin{equation*}
        \int_E d(x,x_0)^2 \rho(dx) < +\infty,
    \end{equation*}
    For any $\rho \in \Pc(E)$ and $\phi \in C_b(E)$, we define $\langle \phi,\rho \rangle := \int_E\phi(x)\rho(dx)$.
    For $E = \R$, we denote by $m(\mu): = \int_{\R}x\mu(dx)$ the mean of the measure,
    
    The Wasserstein $2$-distance $\Wc_2$ between two probability measures $\rho_1$ and $\rho_2$ with $\rho_1, \rho_2 \in \Pc_2(E)$ is defined as
    \begin{equation*}
        \Wc_2(\rho_1,\rho_2)
        ~ := ~
        \bigg(
            \inf_{\rho \in \Gamma(\rho_1,\rho_2)}
            \int_{E \x E}d(e_1,e_2)^2 \rho(de_1,de_2)
        \bigg)^{1/2},
    \end{equation*}
    where $\Gamma(\rho_1,\rho_2)$ is the collection of all probability measures on $E\times E$, such that $\rho(de,E) = \rho_1(de)$ and $\rho(E, de) = \rho_2(de)$.
    For any $\mu \in \Pc_2(\R)$, $\mub := \int_{\R}x \mu(dx)$.
    The space $\Pc_2(E)$ is then naturally equipped with the metric $\Wc_2$.

    The Fourier Wasserstein $2$-distance $\rho_F$ between two probability measures $\rho_1$ and $\rho_2$ with $\rho_1, \rho_2 \in \Pc_2(\R)$ is defined as
    \begin{equation*}
        \rho_F^2(\rho_1,\rho_2)
        ~ := ~
        \int_{\R}\frac{|F_k(\rho_1-\rho_2)|^2}{(1 + |k|^2)^\lambda}dk,
    \end{equation*}
    where for $k \in \R$, $f_k(x) := (2\pi)^{-\frac{d}{2}}e^{i k \cdot x}$, and $F_k(\rho) := \langle f_k, \rho \rangle$. 
    For $\theta_1=(t_1,\mu_1,m_1), \, \theta_2=(t_2,\mu_2,m_2) \in \Theta:=[0,T] \times \Pc_2(\R) \x \R$, define 
    $$
        d_F^2(\theta_1,\theta_2)=|t_1-t_2|^2+|m_1-m_2|^2+\rho_F^2(\mu_1,\mu_2).
    $$
    

    
    \subsection{Measure derivatives}\label{subsec:measure-derivatives}
    We denote by $B_q$ the set of Borel measurable functions on $\R$ with at most quadratic growth.
    \begin{definition}
        \begin{enumerate}[(i)]
            \item The function $u: \Pc_2(\R) \longrightarrow \R$ is said to have a linear functional derivative if there exists
                \begin{equation*}
                    \delta_{\mu}u: \Pc_2(\R) \x \R \longrightarrow \R
                \end{equation*}
            such that $\delta_\mu u$ is continuous for the product topology and
                \begin{itemize}
                    \item for each $\mu \in \Pc_2(\R)$, the mapping $x \longmapsto \delta_{\mu} u(\mu,x) \in B_q$, 
                    \item for all $m_1, m_2 \in \Pc_2(\R)$,
                        \begin{equation}
                            u(m_1) - u(m_2) 
                            ~ = ~
                            \int_0^1 \int_{\R}
                                \delta_{\mu} u(\lambda m_1 + (1 - \lambda)m_2,x)(m_1 - m_2)(dx) d\lambda.
                        \end{equation}
                \end{itemize}
            \item We say a function $u:\Theta \longmapsto \R$ is partial $C^2$-regular
            if
            the function $(t,\mu, m) \in \Theta \longmapsto u(t,\mu, m)$ is $C^{1,2}$ in $(t,m)$ and continuous in all its variables, and $(\pa_x\delta_\mu u, \pa^2_{xx}\delta_\mu u) \in B_q \x B_q$ exist and are continuous in all their variables.
        \end{enumerate}
    \end{definition}

    For any partial $C^2$-regular function $(t,\mu) \longmapsto u(t,\mu)$, we define
    \begin{align*}
        D_\mu u :=& \pa_x\delta_\mu u,
        ~
        D_{x\mu} u := \pa^2_{xx}\delta_\mu u,
        \\
        \Hc u :=& \pa^2_{zz}u(t,y,(I_d + z)_\sharp \mu)|_{z = 0}.
    \end{align*}
    Moreover, if in addition, $D_{\mu\mu}u \in B_q$ exists and continuous in all its variables, then
    \begin{align*}
        \Hc u  =
        \displaystyle
            \int_{\R}\int_{\R}D^2_{\mu\mu}u(\mu,x,y)\mu(dx)\mu(dy) 
                + \int_{\R} D_{x\mu} u(\mu,x)\mu(dx).
    \end{align*}

\section{The common-noise Hessian and the nonlinear transform}\label{sec:change}
\subsection{The original Wasserstein PDE}\label{subsec:original-pde}
The starting point is the following PDE:
\begin{align}\label{eq:origPDE}
    -(\pa_t u + G(\cdot,D_\mu u ,D_{x\mu}u,\Hc_{\sigma_0} u))(t,\mu) = 0,
    \qquad (t,\mu) \in [0,T) \x \Pc_2(\R),
\end{align}
where $G:[0,T) \x \Pc_2(\R) \x B_q \x B_q \x \R \longrightarrow \R$. We also define the weighted Hessian operator as
$$
    \Hc_{\sigma_0} u(t,\mu) = \int_\R\int_\R D_{\mu\mu}u(t,\mu)(x,y)\sigma_0(x,\mu)\sigma_0(y,\mu)\mu(dx)\mu(dy) + 
    \int_\R D_{x\mu}u(t,\mu)(x)\sigma^2_0(x,\mu)\mu(dx).
$$
\subsection{The measure-dependent flow and augmented equation}\label{subsec:augmented-equation}
We first introduce a coefficient that determines the change of variables we will apply.
\begin{assumption}\label{assump:regularity}
The coefficient $\sigma_0: \R \x \Pc_2(\R) \longrightarrow \R$ satisfies the following conditions for some integer $k$:
\begin{enumerate}[(i)]
\item The functions $\sigma_0$ and $D_\mu\sigma_0$ are bounded, Borel measurable, and Lipschitz: there exists a constant $L > 0$ such that for all $(x_1, x_2, \mu,\nu) \in \R \times \R \times \Pc_2(\R) \times \Pc_2(\R)$ and $(y_1, y_2) \in \R \x \R$,
\begin{align*}
    \abs{\sigma_0(x_1, \mu) - \sigma_0(x_2, \nu)}& \leq L \abs{x_1 - x_2} +L|\mu - \nu|_{-\lambda},\\
    \abs{D_\mu\sigma_0(x_1, \mu)(y_1) - D_\mu\sigma_0(x_2, \nu)(y_1)} & \leq L \abs{x_1 - x_2} + \abs{y_1 - y_2} +L|\mu - \nu|_{-\lambda},\\
    \abs{D_{y\mu}\sigma_0(x_1, \mu)(y_1) - D_{y\mu}\sigma_0(x_2, \nu)(y_1)} & \leq L \abs{x_1 - x_2} + \abs{y_1 - y_2} +L|\mu - \nu|_{-\lambda}.
\end{align*}

\item $\sigma_0(\cdot, \mu) \in W^{k+2,\infty}(\R)$ uniformly in $\mu$ and the Lions derivative $D_\mu \sigma_0(\cdot, \mu)(y)$ exists and satisfies $D_\mu \sigma_0(\cdot, \mu)(y) \in W^{k+1,\infty}(\R)$ uniformly in $(y,\mu)$:
\[
\sup_{\mu \in \Pc_2(\R), y \in \R}\bigg(\sum_{i = 0}^{k+2}\norm{\partial_x^i \sigma_0(\cdot, \mu)}_{L^\infty} 
+ \sum_{i = 0}^{k}\norm{\partial_x^iD_\mu\sigma_0(\cdot, \mu)(y)}_{L^\infty} + \norm{\partial_x^iD_{y\mu}\sigma_0(\cdot, \mu)(y)}_{L^\infty}\bigg) \leq C_k,
\]
where $\partial_x^k$ denotes the partial derivative of order $k$ in the first (spatial) argument.
\end{enumerate}
\end{assumption}

In this subsection we derive the PDE satisfied by the extended function $\ub$ defined on the augmented space $[0,T] \x \Pc_2(\R) \x \R$. 
We introduce the forward flow $\psi$ solving
\begin{align}\label{eq:forward-flow}
    \pa_m\psi(x,\mu,m) = \pa_x\psi(x,\mu,m)\sigma_0(x,\mu) + \int_\R D_\mu \psi(x,\mu,m)(y)\sigma_0(y,\mu)\mu(dy),
    \quad \psi(x,\mu,0) = x,
\end{align}
and define the extended function
    $\ub(t,\mu,m) := u(t,\psi(\cdot,\mu,m)_\sharp\mu).$
We aim to express $D_\mu u$, $D_{x\mu}u$, and $\Hc_{\sigma_0}u$ in terms of the derivatives of $\ub$ with respect to $(\mu,m)$.

To reverse the direction, we define the map $m \longmapsto \mu_m := \psi(\cdot,\mu,m)_\sharp\mu$, then observe that setting $\mut_m := \phi(\cdot,\mu,m)_\sharp\mu$ where the backward flow $\phi$ solves
\begin{align}\label{eq:backward-flow}
    \pa_m\phi(x,\mu,m) = -\pa_x\phi(x,\mu,m)\sigma_0(x,\mu) - \int_\R D_\mu \phi(x,\mu,m)(y)\sigma_0(y,\mu)\mu(dy),
    \quad \phi(x,\mu,0) = x,
\end{align}
one has the inversion relations
\begin{align*}
    \psi(\cdot,\mut_m,m)_\sharp\mut_m = \mu,
    \qquad
    \phi(\cdot,\mu_m,m)_\sharp\mu_m = \mu,
    \qquad
    \phi(\cdot,\mu,m) = \psi(\cdot,\mu,-m),
    \qquad
    \mut_m = \mu_{-m}.
\end{align*}
Using the inversion $\mu = \phi(\cdot,\mu_m,m)_\sharp\mu_m$, we can express $u(\mu_m) = \ub(\mu,m)$. More broadly, $u(\mu) = \ub(\mut_m,m)$ with $\mut_m = \phi(\cdot,\mu,m)_\sharp\mu$. Differentiating this identity yields the relations between the Lions derivatives of $u$ and the partial derivatives of $\ub$ (we suppress the time variable in the following calculation for simplicity):
\begin{align*}
    \delta_\mu u(\mu,x) &=
    \delta_\mu \ub(\mut_m,m,\phi(x,\mu,m))
    + \int_\R D_\mu \ub(\mut_m,m,z)\,\delta_\mu \phi(\psi(z,\mu,m), \mu,m)(x)\,\mut_m(dz),
   \\
    D_\mu u(\mu,x) &=
    D_\mu \ub(\mut_m,m,\phi(x,\mu,m))\,\pa_x\phi(x,\mu,m)
    \\&\quad+ \int_\R D_\mu \ub(\mut_m,m,z)\,D_\mu \phi(\psi(z,\mu,m), \mu,m)(x)\,\mut_m(dz),
   \\
    D_{x\mu} u(\mu,x) &=
    D_{x\mu} \ub(\mut_m,m,\phi(x,\mu,m))\,\big(\pa_x\phi(x,\mu,m)\big)^2
    + D_\mu \ub(\mut_m,m,\phi(x,\mu,m))\,\pa^2_x\phi(x,\mu,m) \notag\\
    &\quad + \int_\R D_\mu \ub(\mut_m,m,z)\,D_{x\mu} \phi(\psi(z,\mu,m), \mu,m)(x)\,\mut_m(dz).
\end{align*}
We obtain the expressions needed for the change of variables in the PDE:
\begin{align*}
    D_\mu u(\mu_m,x) &=
    D_\mu \ub(\mu,m,\phi(x,\mu_m,m))\,\pa_x\phi(x,\mu_m,m)
    \\&\quad+ \int_\R D_\mu \ub(\mu,m,z)\,D_\mu \phi(\psi(z,\mu_m,m), \mu_m,m)(x)\,\mu(dz),
    \\
    D_{x\mu} u(\mu_m,x) &=
    D_{x\mu} \ub(\mu,m,\phi(x,\mu_m,m))\,\big(\pa_x\phi(x,\mu_m,m)\big)^2
    \\&\quad+ D_\mu \ub(\mu,m,\phi(x,\mu_m,m))\,\pa^2_x\phi(x,\mu_m,m) \notag
    \\&
    \quad + \int_\R D_\mu \ub(\mu,m,z)\,D_{x\mu} \phi(\psi(z,\mu_m,m), \mu_m,m)(x)\,\mu(dz).
\end{align*}
For the second-order term in $m$, the construction of the flow is designed precisely so that
\begin{align*}
    \Hc_{\sigma_0} u(\mu_m) + \int_\R D_\mu u(\mu_m)(x)\Gamma(x,\mu_m)\mu_m(dx) = \pa^2_{mm}\,\ub(\mu,m),
\end{align*}
where the function $\Gamma: \R \x \Pc_2(\R) \longrightarrow \R$ is given by
$$
    \Gamma(x,\mu) := \partial_x \sigma_0(x,\mu)\,\sigma_0(x,\mu)
+
\int_{\mathbb{R}}
D_\mu \sigma_0(x,\mu)(y)\,\sigma_0(y,\mu)\,\mu(dy).
$$
Substituting these identities into \eqref{eq:origPDE}, we deduce that $\ub$ satisfies the augmented PDE
\begin{align}\label{eq:NewHJB}
    -\big(\pa_t \ub + G^e(\cdot,D_\mu\ub,D_{x\mu}\ub,\pa^2_{mm}\ub)\big)(t,\mu,m) = 0,
    \qquad (t,\mu, m) \in [0,T) \x \Pc_2(\R) \x \R,
\end{align}
where the extended Hamiltonian $G^e$ is defined for $(t,\mu,m, p,q, X) \in [0,T] \x \Pc_2(\R) \x \R \x B_{q} \x B_{q} \x \R$ by
\begin{align*}
    G^e(t,\mu, m,p,q,X)
    :=
    G\bigg(t,\mu_m,\; P(\cdot,\mu,m),\; Q(\cdot,\mu,m),\; X - \int_\R P(x,\mu,m)\Gamma(x,\mu_m)\mu_m(dx)\bigg),
\end{align*}
with $\mu_m = \psi(\cdot,\mu,m)_\sharp\mu$ and
\begin{align*}
    P(x,\mu,m) &:= p(\phi(x,\mu_m,m))\,\pa_x\phi(x,\mu_m,m)
        + \int_\R p(z)\,D_\mu \phi(\psi(z,\mu_m,m), \mu_m,m)(x)\,\mu(dz),
    \\
    Q(x,\mu,m) &:= q(\phi(x,\mu_m,m))\,\big(\pa_x\phi(x,\mu_m,m)\big)^2
        + p(\phi(x,\mu_m,m))\,\pa^2_x\phi(x,\mu_m,m) \notag\\
    &\qquad + \int_\R p(z)\,D_{x\mu} \phi(\psi(z,\mu_m,m), \mu_m,m)(x)\,\mu(dz).
\end{align*}

\subsection{Viscosity solutions for the augmented equation}\label{subsec:augmented-viscosity}
We now introduce the notion of viscosity solution for the augmented PDE. The following definitions are natural extensions of the standard viscosity solution concept to the present infinite-dimensional setting.

\begin{definition}\label{def:jets}
    Let $\ub:[0,T] \x \Pc_2(\R) \x \R \longrightarrow \R$ be a locally bounded function. For $\theta = (t,\mu,m) \in [0,T) \x \Pc_2(\R) \x \R$, the (partial) second-order superjet $J^{2,+}\ub(\theta) \subset \R \x B_{q} \x B_{q} \x \R$ is defined by
    \begin{align*}
        J^{2,+}\ub(\theta)
        :=
        \big\{(\pa_t\varphi,  D_\mu \varphi, D_{x\mu}\varphi, \pa^2_{mm}\varphi)(\theta): \;
        & \ub - \varphi \text{ has a local maximum at } \theta,\\
        & \varphi \text{ is partial } C^2\text{-regular}\big\},
    \end{align*}
    the second-order subjet is $J^{2,-}\ub(\theta) := -J^{2,+}(-\ub)(\theta)$, and the closures $\Jb^{2,+}\ub(\theta)$, $\Jb^{2,-}\ub(\theta)$ are defined analogously by taking limits of jets at approximating points $\theta_n \to \theta$ with $\ub(\theta_n) \to \ub(\theta)$.
\end{definition}

\begin{definition}\label{def:viscosity}
    A function $u: [0,T] \x \Pc_2(\R) \to \R$ is a viscosity subsolution (resp.\ supersolution) of \eqref{eq:origPDE} if the extended function
    \begin{align*}
        \bar u(t,\mu,m) := u(t, \psi(\cdot,\mu,m)_{\sharp}\mu)
    \end{align*}
    is locally bounded and for every $(t,\mu,m) \in [0,T) \x \Pc_2(\R) \x \R$,
    \begin{equation*}
        -b - G^e(t,\mu,m, p, q, X) \le 0 \quad (\text{resp. } \ge 0),
    \end{equation*}
    for all $(b, p, q, X) \in J^{2,+}\,\bar u(t,\mu, m)$ (resp.\ $J^{2,-}\,\bar u(t,\mu, m)$).
\end{definition}

\section{Ishii's lemma and comparison}\label{sec:comparison}

    We equip the space $[0,T] \x \Pc_2(\R) \x \R $ with the product topology, where $\Pc_2(\R)$ is equipped with the $W_1$-topology. Define an auxiliary function 
    \begin{align}\label{eq:vartheta}
        \vartheta: [0,T] \x \Pc_2(\R) \x \R \ni \theta = (t,\mu,m) \longmapsto  e^{-Lt}\bigg(e^{C_*\sqrt{1+m^2}}~+\int_{\R} |x|^2 \, \mu(dx)\bigg),
    \end{align}    
    for some fixed strictly positive constant $C_*$ and $L$. The following lemma has been established in \cite{BEHZ}.

  \begin{lemma}[Ishii's Lemma]\label{lemm:ishii}
    Suppose that $u, -v: [0,T] \x  \Pc_2({\R}) \rightarrow \R$ are bounded upper-semicontinuous functions.   
    For any $\delta>0$, introduce 
    \begin{align*}
        &\tilde u: [0,T] \times \Pc_2(\R) \times \R  \ni (t, \mu, m) \mapsto u(t, \psi(x,\mu,m)_\sharp\mu)- \delta \vartheta(t, \mu, m), \\
        &\tilde v: [0,T] \times \Pc_2(\R) \times \R  \ni (s, \nu, n) \mapsto v(s, \psi(x,\nu,n)_\sharp\nu)+ \delta \vartheta(s, \nu, n).
    \end{align*}
    Then there exists  a local maximum $ (\theta^*,\iota^*)=((t^*,\mu^*,m^*),(s^*,\nu^*,n^*))$ of 
            \begin{align}\label{eq:ishiimax}
             (\theta,\iota) \longmapsto \tilde u(\theta)- \tilde v(\iota) -\frac{1}{2\e} \left(|t-s|^2+|m-n|^2+\rho_F^2(\mu,\nu) \right) .
             \end{align}
            Assume $ \theta^*, \iota^*$ are in the interior $[0,T) \times \Pc_2(\R) \x \R$, and denote
        \begin{align}
            \Lc(\eta,\mu,\nu):=&2\int \frac{Re(F_k(\eta)(F_k(\mu)-F_k(\nu))^*)}{(1+|k|^2)^\lambda}dk, \label{eq:mathcalL}\\
\Phi(\mu):=&2\rho^2_F(\mu,\mu^*)+\Lc(\mu,\mu^*,\nu^*)\label{eq:defpsi},\\
            \Psi(\nu):=& 2\rho^2_F(\nu,\nu^*)-\Lc(\nu,\mu^*,\nu^*),\notag
        \end{align}
        where $F_k^*(\mu), F_k^*(\nu)$ are the conjugates of $F_k(\mu), F_k(\nu)$, respectively.
     Then for any $\a>0$, there exist $X^*,Y^*
    \in \R$ such that 
    \begin{align*}
     & \left(\frac{1}{\e}(t^*-s^*), 
        \frac{D_\mu \Phi(\mu^*)(\cdot)}{2\e},
        \frac{D_{x\mu} \Phi(\mu^*)(\cdot)}{2\e},X^* \right) 
        \in \bar J^{2,+} \tilde u(\theta^*), \\
    & \left(\frac{1}{\e}(t^*-s^*), 
        -\frac{D_\mu \Psi(\nu^*)(\cdot)}{2\e},
        -\frac{D_{x\mu} \Psi(\nu^*)(\cdot)}{2\e},-Y^* \right) 
        \in \bar J^{2,-} \tilde v (\iota^*),
    \end{align*}
        as well as 
    \begin{align*}
        -\left(\frac{1}{\alpha}+\frac{2}{\e}\right)I_{2d}\leq 
        &\begin{pmatrix}
            X^*&0\\0& Y^*
        \end{pmatrix}
        \leq 
        \left(\frac{1}{\e}+\frac{2\alpha}{\e^2} \right)
        \begin{pmatrix}
            I_{d}&-I_{d}\\-I_{d}&I_{d}
        \end{pmatrix}.
    \end{align*}
      
    \end{lemma}

    We impose the following assumption on $G$ for the comparison principle.
    \begin{assumption}\label{ass:comparison}
     \begin{itemize}
       \item[(i)] Assume that the extended Hamiltonian $G^e$ satisfies the following Lipschitz condition: for any $\theta = (t,\mu,m) \in [0,T] \x \Pc(\R) \x \R$, $p_1,p_2 \in B_{q}$, and $q_1,q_2 \in B_{q}$,
        \begin{align*}
            |G^e(\theta,p_1,q_1,X_1)
            - G^e(\theta,p_2,q_2,X_2)|
            \le 
            L_G
            &\bigg(1 +
            \int_{\R}|x|^2\mu(dx)\bigg)e^{C_*|m|}
            \\&
            \Big(|p_1 - p_2|_q + |q_1 - q_2|_q  + |X_1 - X_2| \Big),
        \end{align*}
        for some positive constant $L_G$.
      
        \item[(ii)]
        There exists a modulus of continuity $\omega_G$ such that for all
        $\e > 0$, 
        $(\theta = (t,\mu, m),\iota = (s,\nu,n)) \in \big([0,T) \x \Pc_2(\R) \x \R\big)^2$, we have the inequality
        \begin{align*}
            &~
            G^e\Big(\theta, 
                \nabla\kappa,
                \nabla^2\kappa,
                X
            \Big)
            -
            G^e\Big(\iota, 
                \nabla\kappa,
                \nabla^2\kappa,
                -Y
            \Big) 
            \\ \le & ~
            \omega_G\bigg(\frac{1}{\e}d^2_F(\theta,\iota) + d_F(\theta,\iota)\bigg)  
            (e^{C_*|m|} + e^{C_*|n|}) \bigg(1 + \int_{\R}|x|^2(\mu + \nu)(dx)\bigg),
        \end{align*}
        where the function $\kappa$ on $\R$ is defined as
        \begin{equation*}
            \kappa(x)
            ~ := ~
            \frac{1}{\e}\int_{\R}
            \frac{Re(F_k(\mu - \nu)f^*_k(x))}
            {(1 + |k|^2)^\lambda}
            dk,
           \quad
            x \in \R,
        \end{equation*}
        whenever $X \leq -Y$, $X,Y \in \R $.
        \end{itemize}
    \end{assumption}

    \begin{theorem}[Comparison Principle]\label{thm:comparison}
            Assume that $u, -v: [0,T] \times \Pc_2({\R})  \rightarrow \R$ are bounded upper-semicontinuous functions, and $u$ (resp. $v$) is a viscosity subsolution (resp. supersolution) of the equation
    \begin{align}
        -\pa_t u(t,\mu)
        =
        G(t,\mu, D_\mu u (t,\mu),D_{x\mu}u(t,\mu),\Hc_{\sigma_0}u(t,\mu)).
    \end{align}
        Then, under Assumption \ref{ass:comparison}, $u(T,\cdot) \le v(T,\cdot)$ implies that $u \le v$ for all $(t,\mu) \in [0,T] \x \Pc_2({\R})$.
    \end{theorem}

    \begin{proof}


\medskip
 \noindent  \emph{Step 1: Preliminaries.} 
 Recall that  $\ub(t,\mu,m) := u(t,\psi(x,\mu,m)_\sharp\mu)$, $\vb(t,\mu,m) := v(t,\psi(x,\mu,m)_\sharp\mu)$. We prove the result by contradiction. Assume that
        \begin{equation*}
            \sup_{(t,\mu,m) \in [0,T] \x \Pc_2({\R}) \x \R}(\ub - \vb)(t,\mu,m) \geq 4r,
        \end{equation*}
        where $r$ is a strictly positive constant.
        Then we choose a $h > 0$ small enough depending on $r$ such that
        $\ub_h:= \ub - h(T - t + 1)$ satisfies
        \begin{equation*}
            \sup_{(t,\mu,m) \in [0,T] \x \Pc_2({\R}) \x \R}(\ub_h - \vb)(t,\mu,m) \ge 3r > 0.
        \end{equation*}
        We then verify that $\ub_h$ is still a viscosity subsolution of the equation \eqref{eq:NewHJB}. Suppose $\phi$ is a $C^2$-regular test function such that $\ub_h - \phi$ has a local maximum at $(t^*,\mu^*,m^*)$. Equivalently, $\ub - (\phi + h(T - t + 1))$ has a local maximum at $(t^*,\mu^*,m^*)$. 
        We write $\phi + h(T - t + 1)$ as $\phi_h$ for short.
        Then, the viscosity property of $u$ gives that, for any $(t^*,\mu^*,m^*)$,
        \begin{align*}
            &
            -\left(\pa_t \phi_h
                +
                G^e(\cdot, D_\mu \phi_h,D_{x\mu} \phi_h, \pa^2_{mm} \phi_h)\right)
            (t^*, \mu^*, m^*)
            \\  ~ = &~
            -\left(\pa_t \phi_h
                +
                G^e(\cdot, D_\mu \phi,D_{x\mu} \phi, \pa^2_{mm} \phi)\right)(t^*, \mu^*, m^*)
            \\ ~ = & ~
             -\left(\pa_t \phi
                +
                G^e(\cdot, D_\mu \phi,D_{x\mu} \phi, \pa^2_{mm} \phi)\right)(t^*, \mu^*, m^*)
                - h
            ~ \le ~ - h 
            ~ <   ~ 0.
        \end{align*}
        
        \medskip
                
       \noindent \emph{Step 2: Doubling the variables.}
        Recall from \eqref{eq:vartheta} that, for each $\theta = (t, \mu, m) \in [0,T] \times \Pc_2({\R}) \x \R $,
        $\vartheta(\theta)= e^{-Lt}(e^{C\sqrt{1+m^2}}+ \int_{{\R}}|x|^2\mu(dx))$, and $\vartheta$ is lower-semicontinuous. 
        
        Now for any $\e,\delta > 0$ and $(\theta, \iota) \in \left( [0,T] \x \Pc_2({\R}) \x \R  \right)^2 $, we define
        \begin{equation*}
            H_\e^\delta(\theta, \iota) 
            ~ := ~
            \ub_h(\theta) - \vb(\iota) - \frac{1}{2\e}d^2_F(\theta,\iota) - \delta\big(\vartheta(\theta) + \vartheta(\iota)\big),
        \end{equation*}
        We show that the supremum of $H^\delta_\e$ over all admissible $(\theta,\iota)$ is attained.
        
        In fact, there exists an element $\theta_0 \in [0,T]  \x \Pc_2({\R}) \x \R$, such that
        \begin{equation*}
            H_\e^\delta(\theta_0,\theta_0) + 2\delta\vartheta(\theta_0)
            ~ = ~
            \ub_h(\theta_0) - \vb(\theta_0)
            ~ \ge ~ 2r > 0,
        \end{equation*}
        If $\delta < \frac{r}{2\vartheta(\theta_0) + 1}$, then
        \begin{align*}
            \sup_{(\theta, \iota) \in ([0,T]   \x  \Pc_2({\R})\x \R)^2}
            H_\e^\delta(\theta, \iota)
            ~ \ge ~ 
            H_\e^\delta(\theta_0,\theta_0)
            ~ \ge ~
            r > 0.
        \end{align*}
        Take a sequence $\{(\theta_k,\iota_k)\}_{k \in \N_+}$ with $H_\e^\delta(\theta_k,\iota_k) > 0$ such that
        \begin{equation*}
            \lim_{k \to \infty}H_\e^\delta(\theta_k,\iota_k) 
            ~ = ~
            \sup_{(\theta, \iota) \in ([0,T]   \x  \Pc_2({\R})\x \R)^2}
            H_\e^\delta(\theta, \iota).
        \end{equation*}
        For each $k \in \N_+$, we have
        \begin{align*}
            \delta\big(\vartheta(\theta_k) + \vartheta(\iota_k)\big)
            &~ \le ~
            \ub_h(\theta_k) - \vb(\iota_k) - \frac{1}{2\e}d^2_F(\theta_k,\iota_k) \\
            &~ \le ~
            \ub_h(\theta_k) - \vb(\iota_k)
            ~ \le ~
            2M,
        \end{align*}
        where $M$ is a common upper bound for $u$ and $-v$.
        Since, for any $c > 0$, the sublevel set $\{\theta \in [0,T]  \x \Pc_2({\R})  \x \R: \vartheta(\theta) \le c\}$ is compact,  
        without loss of generality we can assume that $\{(\theta_k,\iota_k)\}_{k \in \N_+}$ is
        a convergent sequence in $[0,T] \x \Pc_2({\R}) \x \R$, and 
        $$
            \lim_{k \to \infty}
            (\theta_k, \iota_k)
            ~ = ~
            (\theta^\delta_\e,\iota^\delta_\e),
        $$
        for some $(\theta^\delta_\e,\iota^\delta_\e) \in ([0,T] \x \Pc_2({\R}) \x \R)^2$.
        As a consequence, the upper-semicontinuity of $H^\delta_\e $ yields that
        \begin{align*}
            &
            H_\e^\delta(\theta^\delta_\e,\iota^\delta_\e)
            ~ = ~
            \sup_{(\theta, \iota) \in ([0,T]  \x  \Pc_2({\R}) \x \R)^2}
            H_\e^\delta(\theta, \iota)
            ~>~ 0,
            \\~
            &
            \delta\big(\vartheta(\theta^\delta_\e) 
             + \vartheta\big(\iota^\delta_\e)\big)
             ~ \le ~
             \varliminf_{k \to \infty}
                \delta\big(\vartheta(\theta_k) 
                + \vartheta\big(\iota_k)\big)
             ~ \le ~
             2M.
        \end{align*}
        Again, without loss of generality we can assume that $\{(\theta^\delta_\e,\iota^\delta_\e)\}_{\e \in \N_+}$ is
        a convergent sequence in $[0,T] \x \Pc_2({\R}) \x \R$, and 
        $$
            \lim_{\e \to 0+}
            (\theta^\delta_\e,\iota^\delta_\e)
            ~ = ~
            (\theta^\delta,\iota^\delta),
        $$
        for some $(\theta^\delta,\iota^\delta) \in ([0,T] \x \Pc_2({\R}) \x \R)^2$.

        We next prove that 
        \begin{equation}\label{eq:penalGoZero}
            \lim_{\e \to 0}
            \frac{1}{2\e}d^2_F(\theta^\delta_\e,\iota^\delta_\e)
            ~ = ~ 0.
        \end{equation}
        Indeed,
        \begin{align*}
             \varlimsup_{\e \to 0+}
             \frac{1}{2\e}d^2_F(\theta^\delta_\e,\iota^\delta_\e)
             & ~ =  ~
             \varlimsup_{\e \to 0+}
             \Big(- H_\e^\delta(\theta^\delta_\e,\iota^\delta_\e)
             - \delta\big(\vartheta(\theta^\delta_\e) 
                + \vartheta\big(\iota^\delta_\e)\big)
             + \ub_h(\theta^\delta_\e) - \vb(\iota^\delta_\e)\Big)
             \\ & ~ \le ~ 
              \varlimsup_{\e \to 0+}
              \Big(\ub_h(\theta^\delta_\e) - \vb(\iota^\delta_\e)\Big)
             ~ \le ~
             2M,
        \end{align*}
        and obtain that
        \begin{equation*}
            \varlimsup_{\e \to 0+} 
            d^2_F(\theta^\delta_\e,\iota^\delta_\e)
            ~ = ~
            0,
            ~\mbox{i.e.}~
            d^2_F(\theta^\delta,\iota^\delta)
            ~ = ~ 0,
            ~\mbox{or}~
            \theta^\delta 
            ~ = ~
            \iota^\delta.
        \end{equation*}
        Then it follows that
        \begin{align*}
            \varlimsup_{\e \to 0+}
             \frac{1}{2\e}d^2_F(\theta^\delta_\e,\iota^\delta_\e)
             & ~ =  ~
             \varlimsup_{\e \to 0+}
             \Big(- H_\e^\delta(\theta^\delta_\e,\iota^\delta_\e)
             - \delta\big(\vartheta(\theta^\delta_\e) 
                + \vartheta\big(\iota^\delta_\e)\big)
             + \ub_h(\theta^\delta_\e) - \vb(\iota^\delta_\e)\Big)
             \\ & ~ \le ~ 
             \varlimsup_{\e \to 0+}
             \Big(- \sup_{\theta \in [0,T] \x \Pc_2(\R) \x \R}H_\e^\delta(\theta,\theta)
             - \delta\big(\vartheta(\theta^\delta_\e) 
                + \vartheta\big(\iota^\delta_\e)\big)
             + \ub_h(\theta^\delta_\e) - \vb(\iota^\delta_\e)\Big)
             \\ & ~ \le ~ 
             - \sup_{\theta \in [0,T] \x \Pc_2(\R) \x \R}H_\e^\delta(\theta,\theta)
             - \delta\big(\vartheta(\theta^\delta) 
                + \vartheta\big(\iota^\delta)\big)
             + \ub_h(\theta^\delta) - \vb(\iota^\delta)
             \\ & ~ \le ~ 
             - \sup_{\theta \in [0,T] \x \Pc_2(\R) \x \R}H_\e^\delta(\theta,\theta)
             +  H_\e^\delta(\theta^\delta,\theta^\delta)
             ~ \le ~ 0,
        \end{align*}
        where the second inequality follows from the fact that $\sup_{\theta \in [0,T] \x \Pc_2(\R) \x \R}H_\e^\delta(\theta,\theta)$ is independent of $\e $, and
        that $\ub_h$, $-\vb$, and $-\vartheta$ are upper semicontinuous. 
        
        Moreover, it is clear that $t^\delta \neq T$, where $\theta^\delta = (t^\delta,\mu^\delta, m^\delta)$.
        Otherwise, it follows that
        \begin{equation*}
            0 < r \le
            \varlimsup_{\e \to 0+}
            H^\delta_\e (\theta^\delta_\e,\iota^\delta_\e)
            ~ \le ~
            \varlimsup_{\e \to 0+}
            \ub_h(\theta^\delta_\e) - \vb(\iota^\delta_\e)
            ~ \le ~
            \ub_h(\theta^\delta) - \vb(\iota^\delta)
            ~ \le ~ 
            0.
        \end{equation*}
\medskip
        \noindent \emph{Step 3: Contradiction by viscosity property and estimation.} Define 
        $
             \ut_h := \ub_h - \delta\vartheta,
             ~
             \vt := \vb + \delta\vartheta.
        $ 
        Without loss of generality, we assume that 
        $(\theta^\delta_{\e},\iota^\delta_{\e})$
        is a strict global maximum of 
        $$
            (\theta,\iota) \longmapsto H^\delta_\e (\theta,\iota) 
        = \ut_h(\theta) - \vt(\iota) - \frac{1}{2\e}d^2_F(\theta,\iota).
        $$
        Then by Ishii's lemma \ref{lemm:ishii}, it follows that 
        for any $\e>0$, there exist $X^*,Y^* \in \R$ such that 
        \begin{align*}
         & \left(\frac{1}{\e}(t^*-s^*), 
            \frac{D_\mu \Phi(\mu^*)(\cdot)}{2\e},
            \frac{D_{x\mu} \Phi(\mu^*)(\cdot)}{2\e},X^* \right) 
            \in \bar J^{2,+} \ut_h(\theta^*), \\
        & \left(\frac{1}{\e}(t^*-s^*),
            -\frac{D_\mu \Psi(\nu^*)(\cdot)}{2\e},
            -\frac{D_{x\mu} \Psi(\nu^*)(\cdot)}{2\e},-Y^* \right)
            \in \bar J^{2,-} \vt (\iota^*),
            \end{align*}
            as well as 
        \begin{align*}
            -\left(\frac{1}{\alpha}+\frac{2}{\e}\right)I_{2d}\leq 
            &\begin{pmatrix}
                X^*&0\\0& Y^*
            \end{pmatrix}
            \leq 
            \left(\frac{1}{\e}+\frac{2\alpha}{\e^2} \right)
            \begin{pmatrix}
                I_{d}&-I_{d}\\-I_{d}&I_{d}
            \end{pmatrix}.
        \end{align*}
        For clarity, we record the notation
        \begin{align*}
            &
            \theta^\delta_\e 
            ~ = ~
            (t^\delta_\e,\mu^\delta_\e,m^\delta_\e),
            ~
            \iota^\delta_\e 
            ~ = ~
            (s^\delta_\e,\nu^\delta_\e,n^\delta_\e), \\ &
            q(m) := e^{C\sqrt{1+m^2}}, 
            \\ &
            \kappa^\delta_\e (x)
            ~ := ~
            \frac{1}{\e}\int_{\R}
            \frac{Re(F_k(\mu^\delta_\e - \nu^\delta_\e)f^*_k(x))}
            {(1 + |k|^2)^\lambda}
            dk,
             ~
            x \in \R.
        \end{align*}
       Direct calculation yields that for $\theta = (t,\mu, m) \in [0,T] \x \Pc_2(\R)  \x \R$,
        \begin{align*}
            &
            D_\mu\vartheta(\theta,x)
            ~ = ~
             2e^{-Lt}x,
            \quad
            D_{x\mu}\vartheta(\theta,x)
            ~ = ~
             2e^{-Lt},
            \quad
            \nabla_m^2\vartheta(\theta)
            ~ = ~
             e^{-Lt} \nabla^2 q(m),
            \\
            &
            \frac{D_\mu \Phi(\mu^*)(\cdot)}{2\e}
            ~ = ~
            -\frac{D_\mu \Psi(\nu^*)(\cdot)}{2\e}
            ~ = ~
            \nabla\kappa^\delta_\e(\cdot),
            \quad
            \frac{D_{x\mu} \Phi(\mu^*)(\cdot)}{2\e}
            ~ = ~
            -\frac{D_{x\mu} \Psi(\nu^*)(\cdot)}{2\e}
            ~ = ~
            \nabla^2\kappa^\delta_\e(\cdot).
        \end{align*}
        For convenience, we denote by
        \begin{align*}
            &\alpha^\delta_\e 
            ~:= ~
            \bigg(
                \nabla\kappa^\delta_\e + 2\delta e^{-Lt^\delta_\eps}I_d,
                \nabla^2\kappa^\delta_\e + 2\delta e^{-Lt^\delta_\eps},
               X^* + \delta e^{-Lt^\delta_\eps}\nabla^2q(m^\delta_\eps)
            \bigg),
            \\~&
            \beta^\delta_\e 
            ~:= ~
            \bigg(
                \nabla\kappa^\delta_\e - 2\delta e^{-Ls^\delta_\eps}I_d,
                \nabla^2\kappa^\delta_\e - 2\delta e^{-Ls^\delta_\eps},
               -Y^* - \delta e^{-Ls^\delta_\eps}\nabla^2q(n^\delta_\eps)
            \bigg),
        \end{align*}
        Then,  the mapping
        $ \theta \longmapsto (D_{\mu}\vartheta, D_{x \mu} \vartheta, \nabla^2_m \vartheta)(\theta) $ is continuous, and it follows from the linearity and the definition of the closure of the jets that
        \begin{align*}
            &
            \bigg(
                \frac{1}{\e}(t^\delta_\e - s^\delta_\e) -L\delta\vartheta(\theta^\delta_\e),
                \alpha^\delta_\e 
            \bigg)
            \in
            \Jb^{2,+}\ub_h(\theta^\delta_\e),
            ~
            \bigg(
                \frac{1}{\e}(t^\delta_\e - s^\delta_\e) +L\delta\vartheta(\iota^\delta_\e),
                \beta^\delta_\e 
            \bigg)
            \in
            \Jb^{2,-}\vb(\theta^\delta_\e),
        \end{align*}
        By viscosity property of $\ub_h$, and $\vb$, one has that
        \begin{align*}
            -
            \frac{1}{\e}(t^\delta_\e - s^\delta_\e)
            +
            L\delta\vartheta(\theta^\delta_\e)
            -  
            G^e\Big(\theta^\delta_\e, \alpha^\delta_\e 
            \Big) 
            ~ \le 
            -h,
            \quad
            -
            \frac{1}{\e}(t^\delta_\e - s^\delta_\e)
            - 
            L\delta\vartheta(\iota^\delta_\e)
            -
            G^e\Big(\iota^\delta_\e, \beta^\delta_\e 
            \Big) 
            \ge 
            0,
        \end{align*}
        i.e.,
        \begin{align*}
            h \le G^e\Big(\theta^\delta_\e, 
                \alpha^\delta_\e 
            \Big)
            -
            G^e\Big(\iota^\delta_\e, 
                \beta^\delta_\e 
            \Big)
            -L\delta \vartheta(\theta^\delta_\e) -  L\delta\vartheta(\iota^\delta_\e),
        \end{align*}
        We also have the estimate
        \begin{align*}
            &
            G^e\Big(\theta^\delta_\e, 
                \alpha^\delta_\e 
            \Big)
            -
            G^e\Big(\iota^\delta_\e, 
                \beta^\delta_\e 
            \Big) 
            \\ \le ~ &
            G^e\Big(\theta^\delta_\e, 
                \nabla\kappa^\delta_\e,
                \nabla^2\kappa^\delta_\e,
                X^*
            \Big)
            -
            G^e\Big(\iota^\delta_\e, 
                \nabla\kappa^\delta_\e,
                \nabla^2\kappa^\delta_\e,
                -Y^*
            \Big) 
                \\ & +
            L_G\bigg(2 + e^{C_*\sqrt{1 + |m^\delta_\eps|^2}} + e^{C_*\sqrt{1 + |n^\delta_\eps|^2}} + \int_{\R}|x|^2(\mu^\delta_\e + \nu^\delta_\e)(dx)
            \bigg)
            (4 + 2(C+C^2)e^{C})(e^{-Lt^\delta_\eps} +e^{-Ls^\delta_\eps})\delta 
            \\ \le ~ &
            \omega_G\bigg(\frac{1}{\e}d^2_F(\theta^\delta_\e,\iota^\delta_\e) + d_F(\theta^\delta_\e,\iota^\delta_\e)\bigg)
            \bigg(2 + \int_{\R}|x|^2(\mu^\delta_\e + \nu^\delta_\e)(dx)
            \bigg)\Big(e^{C_*\sqrt{1 + |m^\delta_\eps|^2}} +  e^{C_*\sqrt{1 + |n^\delta_\eps|^2}}\Big)
                \\ & +
                L_G(4 + 2(C+C^2)e^{C})(e^{-Lt^\delta_\eps} +e^{-Ls^\delta_\eps})\delta  \bigg(2  + \int_{\R}|x|^2(\mu^\delta_\e + \nu^\delta_\e)(dx)
            \bigg)\Big(e^{C_*\sqrt{1 + |m^\delta_\eps|^2}} +  e^{C_*\sqrt{1 + |n^\delta_\eps|^2}}\Big),
        \end{align*}
        where the first inequality follows from Assumption \ref{ass:comparison} (i), and the second follows from Assumption \ref{ass:comparison} (ii).
        Letting $\e \to 0$ and choosing $L$ sufficiently large, we obtain
        \begin{align*}
             h 
            ~ \le &~
            8(2 + (C+C^2)e^C) L_G
            \bigg(1 + \int_{\R}|x|^2\mu^\delta(dx)
            \bigg)e^{C_*\sqrt{1 + |m^\delta|^2}} e^{-Lt^\delta} \delta  
            \\ &- 2L \delta e^{-Lt^\delta}\int_{\R}|x|^2\mu^\delta(dx)e^{C_*\sqrt{1 +  |m^\delta|^2}}
            \\ ~ \le &~
            8(2 + (C+C^2)e^C)\delta.
        \end{align*}
        Finally, letting $\delta$ go to $0$, one gets the desired contradiction. 
    \end{proof}

\section{Filtering and mean field control applications}\label{sec:applications}

\subsection{Controlled filtering with state- and law-dependent common noise}\label{sec:appli}
As an application, we show that the value function of a stochastic control problem with partial observation is the unique viscosity solution of a parabolic equation. This section continues the filtering/Bellman program of Bandini, Cosso, Fuhrman, and Pham~\cite{MR3907014}, now in a setting where the common-noise coefficient depends on the state and its conditional law. Relative to the comparison results in \cite{BEHZ} and \cite{BCEQTZ25}, the new feature is this state- and law-dependent common-noise coefficient.

Let $B,W$ be two independent Brownian motions on a filtered probability space $(\Omega, \mathcal{F},\mathbb P)$, and let $A \subset \R$ be a compact control set. Consider coefficients $b:\R \x \Pc(\R)\times A \to \R$, $\sigma:\R\x \Pc(\R) \times A \to \R$, $\sigma_0: \R \x \Pc_2(\R) \to \R$, and $\tilde\sigma:A \to \R$.
The space of probability measures $\Pc(\R)$ is equipped with the Fourier-Wasserstein distance $\rho_F$.
We consider the following McKean--Vlasov stochastic differential equation:
\begin{align*}
dX^{t,\mu,\alpha}_s&=b(X^{t,\mu,\alpha}_s,m^{t,\mu,\alpha}_s,\alpha_s) \, ds + \sigma(X^{t,\mu,\alpha}_s,m^{t,\mu,\alpha}_s,\alpha_s) \, dB_s + \sigma_0(X^{t,\mu,\alpha}_s, m^{t,\mu,\alpha}_s)\tilde\sigma(\alpha_s) \, dW_s, \, \text{$t \leq s \leq T$}, \\
X_t^{t,\mu,\alpha}&=\xi,
~
m^{t,\mu,\alpha}_s := \Lc(X^{t,\mu,\alpha}_s \, |\,\Fc^W_s),
\end{align*}
where $\xi$ is independent of $B,W$ with distribution $\mu \in \Pc_2(\R)$ and $\alpha: \Omega \x [0,T] \rightarrow A$ is an admissible control adapted only to the filtration generated by $W$. Since $\xi$ is independent of $B,W$, standard conditioning shows that the distribution of $(X_s^{t,\mu,\alpha},\alpha_s)$ is independent of the choice of $\xi$.

Then $m_t$, the conditional law of state $X$ at time $t$, satisfies the equation
\begin{align}\label{eq:conditionlaw}
    d \langle m^{t,\mu,\alpha}_s, f \rangle &=  \langle m^{t,\mu,\a}_s, L^{\alpha_s,\cdot} f \rangle \,ds + \langle m^{t,\mu,\a}_s, M^{\a_s,\cdot} f \rangle \, dW_s, \quad \text{ $t \leq s \leq T$}  \\
    m^{t,\mu,\alpha}_t &= \mu, \notag
\end{align}
where $f:\R \x \Pc(\R) \longrightarrow \R$ is any $C^2$ test function and
\begin{align*}
    L^{a,\nu} f(\cdot):=&  b(\cdot,\nu,a) \pa_xf(\cdot) +\frac{1}{2} Tr \left(  \big(\sigma(\cdot,\nu,a)\sigma(\cdot,\nu,a)  +\sigma_0(\cdot,\nu)\tilde\sigma(a)\sigma_0(\cdot,\nu)\tilde\sigma(a)\big) \pa_x^2  f(\cdot)  \right),\\
    M^{a,\nu} f(\cdot):=&   \sigma_0(\cdot,\nu)\tilde\sigma(a) \pa_xf(\cdot).
\end{align*}

Take $\mathcal{A}:=\{\alpha=(\alpha_s)_{0\leq s \leq T}: \alpha : \Om \x [0,T] \longrightarrow A \text{ is measurable and $\F^W$ adapted} \}$ to be the set of all admissible controls. Given a running cost $r:\R \times \Pc(\R) \x  A \to \R$, and a terminal cost $l: \R \to \R$, we define the cost of control $\alpha \in \mathcal{A}$,
\begin{align*}
    J(t,\mu, \alpha):=\E \left[\int_t^T r(X^{t,\mu,\alpha}_s, m^{t,\mu,\alpha}_s, \alpha_s) \, ds + l(X^{t,\mu,\alpha}_T)  \right].
\end{align*}
We consider the following optimization problem:
\begin{align*}
    v(t,\mu)=\inf_{\alpha \in \mathcal{A}} J(t,\mu,\alpha).
\end{align*}

\begin{assumption}\label{assumption:application}
\begin{itemize}
    \item[(i)] The functions $b,\sigma,\tilde\sigma, r,l$ are bounded and Lipschitz continuous in their domains uniformly in $a$.
    \item[(ii)] $\sigma$ is uniformly elliptic, i.e., there exists a positive constant $\delta$ such that $ |\sigma(x,\mu,a) \xi |^2  \geq \delta |\xi|^2$ for any $x,a,\xi,\mu$.
    \item[(iii)] The functions $ x\longmapsto \sigma(x,\mu,a) \in W^{\lceil\lambda+3/2\rceil, \infty}(\R)$,
    $ x\longmapsto b(x,\mu,a) \in W^{\lceil\lambda\rceil, \infty}(\R)$,
    $ x\longmapsto r(x,\mu,a) \in H^{\lambda}(\R)$ uniformly in $a \in A$, $\mu \in \Pc(\R)$.
\end{itemize}
\end{assumption}

\begin{remark}
In \cite[Assumption 4.1]{BEHZ}, the uniform ellipticity of $x \mapsto \sigma(x,\mu,a)$ appears to be incompatible with the imposed integrability condition. The argument in \cite{BEHZ}, however, only uses a uniform bound on $x \mapsto \sigma(x,\mu,a)$, as is clear from the last paragraph of the proof of \cite[Proposition 5.6]{BEHZ}. Accordingly, we replace the integrability requirement by the condition
\[
x \mapsto \sigma(x,\mu,a) \in W^{\lceil\lambda+3/2\rceil, \infty}(\R).
\]
The proof in \cite{BEHZ} then carries over under this modified assumption.
\end{remark}

We define, for $(a,\mu,p,q,M) \in A  \times \Pc_2(\R) \times B_{q} \times B_{q} \times \R $,
\begin{align*}
    K(a,\mu,p,q,M):=& \int  r(x,\mu,a)+ b(x,\mu,a) p(x) +\frac{1}{2} q(x) \sigma^2 (x,\mu,a) \, \mu(dx) +\frac{1}{2}\tilde \sigma^2(a) M.
\end{align*}

\begin{theorem}[Viscosity solution]\label{thm:viscosity_property}
    Under Assumptions~\ref{assump:regularity} and~\ref{assumption:application}, the value function is continuous and is the unique viscosity solution of the equation
    \begin{align}\label{eq:filtering}
        -\pa_t v(t,\mu)=& \inf_{a \in A} K(a, \mu, D_{\mu} v(t,\mu), D_{x \mu} v(t,\mu), \mathcal{H}_{\sigma_0}v(t,\mu)) \notag \\
        v(T,\mu)=&\mu(l).
    \end{align}
\end{theorem}

\proof [Proof of Theorem \ref{thm:viscosity_property}]

Under the boundedness and Lipschitz assumptions on $b,\sigma,\tilde \sigma,\sigma_0, r,l$, standard stability estimates for the controlled filter imply that $(t,\mu) \mapsto v(t,\mu)$ is continuous.
An argument parallel to Proposition 6.3 in \cite{DaJaSe23}, combined with the measure-dependent flow $\psi(x,\mu,m)$ and the pushforward representation of the conditional law, shows that the value function is a viscosity solution of \eqref{eq:filtering}.

For uniqueness, we verify that the Hamiltonian
$$G^e: [0,T] \x \Pc_2(\mathbb R) \x \mathbb R \times B_{q} \x B_{q} \x \R \longrightarrow \mathbb R; \, (t,\mu,m,p,q,M) \mapsto \inf_{a \in A} K^e(a,\mu,m,p,q,M)$$ satisfies Assumption~\ref{ass:comparison},
where $$K^e(a,\mu,m,p,q,M):= K\bigg(a,\mu_m,P(\cdot,\mu,m),Q(\cdot,\mu,m),M - \int P(x,\mu,m)\Gamma(x,\mu_m)\mu_m(dx)\bigg).$$
Recall from Section~\ref{sec:change} that $\mu_m = \psi(\cdot,\mu,m)_\sharp\mu$ and for any $(x,\mu,m) \in \R \x \Pc_2(\R) \x \R$,
\begin{align*}
    P(x,\mu,m) := p(\phi(x,\mu_m,m))\pa_x\phi(x,\mu_m,m) ~+& \int_\R p(z)D_\mu \phi(\psi(z,\mu_m,m), \mu_m,m)(x)\mu(dz),
    \\
    Q(x,\mu,m) := q(\phi(x,\mu_m,m))\big(\pa_x\phi(x,\mu_m,m)\big)^2 ~
    +&~  p(\phi(x,\mu_m,m))\pa^2_x\phi(x,\mu_m,m)
        \\& + ~\int_\R p(z)D_{x\mu} \phi(\psi(z,\mu_m,m), \mu_m,m)(x)\mu(dz),
\end{align*}
where $\phi$ is the backward flow solving \eqref{eq:backward-flow}.

\vspace{0.5em}
\noindent \textbf{For Assumption~\ref{ass:comparison} (i)}, it suffices to show that for any fixed $a$
\begin{align}\label{eq:filtering1}
   & \left| G^e(t,\mu,m,p_1,q_1,M_1)- G^e(t,\mu,m,p_2,q_2,M_2)  \right| \\
   \le ~& \left| K^e(a, \mu,m, p_1,q_1,M_1) - K^e(a,\mu,m,p_2,q_2,M_2) \right|  \notag\\
     \leq ~&L_G \bigg(1 +
            \int_{\R}|x|^2 \, \mu(dx)\bigg)e^{C_*\sqrt{1+m^2}}
             \times \Big(|p_1 - p_2|_q + |q_1 - q_2|_q  + |M_1 - M_2| \Big), \notag
\end{align}
where $L_G$ is some positive constant independent of $a$.
Indeed, from the construction of $K$ and the uniform boundedness of $b,\sigma,\tilde\sigma,\pa_x\sigma_0, D_\mu\sigma_0, \sigma_0$, we have
\begin{align*}
    & \left| K^e(a, \mu,m, p_1,q_1,M_1) - K^e(a,\mu,m,p_2,q_2,M_2) \right| \\
     \leq ~ & C \left( \int_{\R} |P_1(x)-P_2(x)| + |Q_1(x)-Q_2(x)| \, \mu_m(dx) + |M_1-M_2| \right) .
\end{align*}
By the linearity of $P,Q$ in $(p,q)$ and the uniform bounds on $\pa_x\phi$, $D_\mu\phi$, $\pa^2_x\phi$, $D_{x\mu}\phi$ provided by Theorems~\ref{thm:spatial-regularity}, together with the change of measure $\mu \mapsto \mu_m$, we obtain
\begin{align*}
    \int_{\R} |P_1 - P_2| \, \mu_m(dx) &\le Ce^{C|m|} \int_{\R} |p_1 - p_2| \, \mu_m(dx) \le Ce^{C|m|}\bigg(1 + \int_{\R}|x|^2\mu_m(dx)\bigg) |p_1-p_2|_q,
\end{align*}
and an analogous estimate for $Q_1-Q_2$. Since $\mu_m = \psi(\cdot,\mu,m)_\sharp\mu$ and $\psi$ has at most exponential growth in $m$ (which can be controlled by $e^{C_*\sqrt{1+m^2}}$ using the $\vartheta$ function \eqref{eq:vartheta}), the second moment of $\mu_m$ is controlled. Hence \eqref{eq:filtering1} is verified.

\vspace{0.5em}
\noindent
\textbf{For Assumption~\ref{ass:comparison} (ii)}, recalling that
\begin{align*}
      \kappa(x)
      ~ = ~
      \frac{1}{\e}\int_{\R}
            \frac{Re(F_k(\mu - \nu)f^*_k(x))}
            {(1 + |k|^2)^\lambda}
            dk,
    \quad
    \mu_m := \psi(\cdot,\mu,m)_\sharp\mu,
    \quad
    \nu_n := \psi(\cdot,\nu,n)_\sharp\nu,
\end{align*}
and that $M \mapsto K(a,\mu,p,q,M)$ is increasing (since $\frac{1}{2}\tilde\sigma^2(a) \ge 0$), it is enough to estimate
\begin{equation}\label{eq: verify Assumption}
    \begin{aligned}
        & K(a, \mu_m, X(\cdot,\mu,m), Y(\cdot,\mu,m), M) - K(a,\nu_n, X(\cdot,\nu,n), Y(\cdot,\nu,n), M)
        \\
         \leq~ &K(a, \mu_m, X(\cdot,\mu,m), Y(\cdot,\mu,m), M) -
         K(a,\nu_m, X(\cdot,\nu,m), Y(\cdot,\nu,m), M)
         \\
        & + K(a,\nu_m, X(\cdot,\nu,m), Y(\cdot,\nu,m), M) - K(a,\nu_n, X(\cdot,\nu,n), Y(\cdot,\nu,n), M)
        \\
        := & K_1(m,\mu,\nu) + K_2(m,n,\nu)
    \end{aligned}
\end{equation}
where, for $(x,\mu,m) \in \R \x \Pc_2(\R) \x \R$, the lift of $\kappa$ through the change of variables is
\begin{align*}
    X(x,\mu,m) :=&~ \nabla\kappa(\phi(x,\mu_m,m))\pa_x\phi(x,\mu_m,m) + \int_\R \nabla\kappa(z)D_\mu \phi(\psi(z,\mu_m,m), \mu_m,m)(x)\mu(dz),
    \\
    Y(x,\mu,m) :=&~ \frac{1}{2}\nabla^2\kappa(\phi(x,\mu_m,m))\big(\pa_x\phi(x,\mu_m,m)\big)^2
        + \frac{1}{2}\nabla\kappa(\phi(x,\mu_m,m))\pa^2_x\phi(x,\mu_m,m) \\
        & + \frac{1}{2}\int_\R \nabla\kappa(z)D_{x\mu} \phi(\psi(z,\mu_m,m), \mu_m,m)(x)\mu(dz).
\end{align*}
By Theorems \ref{thm:spatial-regularity} and \ref{thm:Lipschitz}, we have the estimates
\begin{equation}\label{eq:estimate of Xt,Yt}
    \begin{split}
        \sup_{(x,\mu) \in \R \x \Pc_2(\R)}&\Big(|X(x,\mu,m)| + |Y(x,\mu,m)|\Big) 
        \le \frac{Ce^{C_*|m|}|\mu - \nu|_{-\lambda}}{\e},
        \\
        \sup_{\mu \in \Pc_2(\R)}&\Big(|X(x,\mu,m) - X(y,\mu,m)| + |Y(x,\mu,m) - Y(y,\mu,m)|\Big) \le \frac{Ce^{C_*|m|}|\mu - \nu|_{-\lambda}|x-y|}{\e},
        \\
        \sup_{(x,\mu) \in \R \x \Pc_2(\R)}&\Big(|X(x,\mu,m) - X(x,\mu,n)| + |Y(x,\mu,m) - Y(x,\mu,n)|\Big) 
        \\& \qquad\le \frac{C(e^{C_*|m|}+e^{C_*|n|})|\mu - \nu|_{-\lambda}|m-n|}{\e}.
    \end{split}
\end{equation}

\vspace{0.5em}
\noindent \textbf{Estimate of $K(m,\mu_m,\nu_m)$ in \eqref{eq: verify Assumption}.}
We define the effective drift
\begin{align*}
    \bt(x,\mu,a) :=~& b(x,\mu,a) - \frac{1}{2}\tilde\sigma^2(a)\pa_x\sigma_0(x,\mu)\sigma_0(x,\mu) - \frac{1}{2}\tilde\sigma^2(a)\int_{\R} D_\mu\sigma_0(x,\mu)(y)\sigma_0(y,\mu)\mu(dy).
\end{align*}
By the definition of $K$,
\begin{equation}\label{eq:final1}
    \begin{split}
    &K_1(m,\mu,\nu)
    \\ & = \int_{\R} r(x,\mu_m, a)\,\mu_m(dx) - \int_{\R} r(x,\nu_m,a)\,\nu_m(dx)
        \\ &+\int_{\R} \bt(x,\mu_m, a) X(x,\mu,m)\, \mu_m(dx) -
            \int_{\R} \bt(x,\nu_m, a)  X(x,\nu,m)\, \nu_m(dx)
        \\
        & + \frac{1}{2} \int_{\R} \sigma^2(x,\mu_m, a) Y(x,\mu,m)\, \mu_m(dx) -
            \frac{1}{2} \int_{\R} \sigma^2(x,\nu_m, a)  Y(x,\nu,m)\, \nu_m(dx)
    \\
    \le&  \int_{\R} \big(r(x,\mu_m, a) -r(x,\nu_m,a)\big)\,\nu_m(dx) + \int_{\R} r(x,\mu_m,a)\,(\mu_m - \nu_m)(dx)
        \\ &+\int_{\R} \big(\bt(x,\mu_m, a) - \bt(x,\nu_m, a)
        \big)X(x,\mu,m)\, \nu_m(dx) 
        \\ &+ \int_{\R} \bt(x,\mu_m, a)\Big(X(x,\mu,m)\mu_m(dx) - X(x,\nu,m)\nu_m(dx)\Big)
        \\ &+\int_{\R} \big(\sigma^2(x,\mu_m, a) - \sigma^2(x,\nu_m, a)
        \big)Y(x,\mu,m)\, \nu_m(dx)
        \\ & + \int_{\R} \sigma^2(x,\mu_m, a)\Big(Y(x,\mu,m)\mu_m(dx) - Y(x,\nu,m)\nu_m(dx)\Big)
    \\ \leq&~
    \frac{Ce^{C_*|m|}}{\e}|\mu_m -\nu_m|_{-\lambda}|\mu - \nu|_{-\lambda} + C |\mu_m -\nu_m|_{-\lambda} + \frac{Ce^{C_*|m|}|\mu -\nu|^2_{-\lambda}}{\e}
    \\ &+ \int_{\R} \bt(x,\mu_m, a)\Big(X(x,\mu,m)\mu_m(dx) - X(x,\nu,m)\nu_m(dx)\Big)
    \\ &+ \int_{\R} \sigma^2(x,\mu_m, a)\Big(Y(x,\mu,m)\mu_m(dx) - Y(x,\nu,m)\nu_m(dx)\Big),
    \end{split}
\end{equation}
where the last inequality holds by the Lipschitz property of $r,b,\pa_x\sigma_0,D_\mu\sigma_0,\sigma$ in measure, and estimate \eqref{eq:estimate of Xt,Yt}.

\vspace{0.5em}
\noindent

By the dual form of the negative Sobolev norm and Theorem~\ref{thm:spatial-regularity}, we have
\begin{align*}
    |\mu_m - \nu_m|^2_{-\lambda}
    = &
        \sup_{f \in H^\lambda(\R),|f|_\lambda \le 1}
    \int_\R f(\psi(x,\mu,m))(\mu - \nu)(dx)
    \\ \le &
    \sup_{g \in H^\lambda(\R),|g|_\lambda \le C_\lambda}
    \int_\R g(\mu - \nu)(dx)
    \le~ C_\lambda|\mu - \nu|_{-\lambda}.
\end{align*}
To finish the estimate, we reorganize the remaining terms
\begin{align*}
    &\int_{\R} \bt(x,\mu_m, a)\Big(X(x,\mu,m)\mu_m(dx) - X(x,\nu,m)\nu_m(dx)\Big)
    \\ &+ \int_{\R} \sigma^2(x,\mu_m, a)\Big(Y(x,\mu,m)\mu_m(dx) - Y(x,\nu,m)\nu_m(dx)\Big)
    \\ \le & \int_\R B(x,\mu,m,a)\nabla\kappa(x)(\mu - \nu)(dx) ~+~
        \int_\R A(x,\mu,m,a)\nabla^2\kappa(x)(\mu - \nu)(dx)
        + \frac{Ce^{C_*|m|}|\mu - \nu|^2_{-\lambda}}{\e}
\end{align*}
where
\begin{align*}
    B(x,\mu,m,a) := &
    \bigg(
    \int_\R \Big(\bt(z,\mu_m,a)D_\mu\phi(\cdot,\mu_m,m)(z)
        + \frac{1}{2}\sigma^2(z,\mu_m,a)D_{z\mu}\phi(\cdot,\mu_m,m)(z)\Big) \mu_m(dz)
            \\ &+ 
        \bt(\cdot,\mu_m,a)\pa_x\phi(\cdot,\mu_m,m) 
        + \frac{1}{2}\sigma^2(\cdot,\mu_m,a)\pa^2_x\phi(\cdot,\mu_m,m)
        \bigg) \circ \psi(x,\mu_m,m),
    \\
    A(x,\mu,m,a) := &\frac{1}{2}\sigma^2(\cdot,\mu_m,a)\big(\pa_x\phi(\cdot,\mu_m,m)\big)^2 \circ \psi(x,\mu_m,m),
\end{align*}

By Assumption \ref{assump:regularity} with $k = \lambda$, Assumption \ref{assumption:application}, Theorem \ref{thm:spatial-regularity}, and Theorem 6.1 in \cite{behzadan2021multiplication}, it follows that
\begin{align*}
    \sup_{(a,\mu) \in A \x \Pc_2(\R)}\sum_{i = 0}^{\lceil\lambda\rceil}\|\pa_x^iB(\cdot,\mu,m,a)\|_{L^\infty} + 
   \sum_{i = 0}^{\lceil\lambda+3/2\rceil}\|\pa_x^iA(\cdot,\mu,m,a)\|_{L^\infty}
    \le Ce^{C_*|m|}.
\end{align*}

Applying Proposition~\ref{prop:commutator} with $\eta=\mu-\nu$ and $\kappa= \frac{1}{\e} \mathcal{J}_{2\lambda} \eta$, we obtain
\begin{align*}
    K_1(m,\mu,\nu) \le  \frac{Ce^{C_*|m|}}{\e}|\mu -\nu|_{-\lambda}^2 + C|\mu -\nu|_{-\lambda}-\frac{Ce^{C_*|m|}\d}{4\e} |\mu -\nu|_{1-\lambda}^2.
\end{align*}


\vspace{0.5em}
\noindent
\textbf{Estimate of $K_2(m,n,\nu)$ in \eqref{eq: verify Assumption}.}
Using the Lipschitz property of $r,b,\pa_x\sigma_0,D_\mu\sigma_0,\sigma^2$ and the uniform bounds on $\nabla\kappa, \nabla^2\kappa$, the second term is bounded by
\begin{align*}
     & \int_{\R} \big(r(x,\nu_m,a)-r(x,\nu_n,a)\big) \, \nu_m(dx)
     + \int_{\R} r(x,\nu_n,a) \, (\nu_m - \nu_n)(dx)
        \\ +& \int_{\R} (\bt(x,\nu_m,a)-\bt(x,\nu_n,a)) X(x,\nu,m) \, \nu_m(dx)+ \int_\R \bt(x,\nu_n,a)\big(X(x,\nu,m)-X(x,\nu,n)\big)\,\nu_m(dx)
        \\  +&  \int_\R \bt(x,\nu_n,a)X(x,\nu,n)\,(\nu_m - \nu_n)(dx) +\int_{\R} (\sigma^2(x,\nu_m,a)-\sigma^2(x,\nu_n,a)) Y(x,\nu,m) \, \nu_m(dx)
        \\  +&  \int_\R \sigma^2(x,\nu_n,a)\big(Y(x,\nu,m)-Y(x,\nu,n)\big)\,\nu_m(dx) + \int_\R \sigma^2(x,\nu_n,a)Y(x,\nu,n)\,(\nu_m - \nu_n)(dx)
    \\ \le &~ 
    C|\nu_m - \nu_n|_{-\lambda} + \frac{C(e^{C_*|m|}+e^{C_*|n|})|\nu_m - \nu_n|_{-\lambda}|\mu-\nu|_{-\lambda}}{\e}
        \\&+ \int_{\R}\Big(\bt(\cdot,\nu_n,a)X(\cdot,\nu,n)+\sigma^2(\cdot,\nu_n,a)Y(\cdot,\nu,n)\Big)\circ (\psi(x,\nu_m,m) - \psi(x,\nu_n,n))\nu(dx)
    \\ \le &~ C|\nu_m - \nu_n|_{-\lambda} + \frac{C(e^{C_*|m|}+e^{C_*|n|})|\nu_m - \nu_n|_{-\lambda}|\mu-\nu|_{-\lambda}}{\e} 
        \\&+ \frac{Ce^{C_*|n|}|\psi(x,\nu_m,m) - \psi(x,\nu_n,n)||\mu-\nu|_{-\lambda}}{\e}
     \\ \le &~ C|\nu_m - \nu_n|_{-\lambda} + \frac{C(e^{C_*|m|}+e^{C_*|n|})(|m-n| + |\nu_m - \nu_n|_{-\lambda})|\mu-\nu|_{-\lambda}}{\e} 
    ,
\end{align*}
where the $|\cdot|_{-\lambda}$-distance between $\nu_m$ and $\nu_n$ is controlled by $C|m-n|$ using the Lipschitz property of $\psi$ (this follows from the Lipschitz estimates for the measure-dependent flow established in Section~\ref{sec:flow-estimates} together with the uniform bounds on $D_\mu\psi$ from Theorem~\ref{thm:spatial-regularity}). Therefore, we obtain the bound
\begin{align*}
    K(m,n) \le C\bigg(d_F(\theta,\iota)+\frac{d_F^2(\theta,\iota)}{\e} \bigg),
\end{align*}
where $\theta=(\mu,m)$ and $\iota=(\nu,n)$. Together with the estimate for $K(m,\mu_m,\nu_m)$, this verifies Assumption~\ref{ass:comparison} (ii) and completes the proof.
\endproof

\subsection{Mean field control with common noise}\label{sec:appli_mfc}
As another application, the same method gives the viscosity characterization of a stochastic mean field control problem with common noise, where the common-noise coefficient depends on the conditional law.
This characterization is also relevant for particle approximation questions. In the state-independent common-noise setting, \cite{BayraktarEkrenZhang2025Particle} obtains convergence rates for finite-particle approximations of second-order PDEs on Wasserstein space using comparison principles; related finite-dimensional approximation, master-equation, and convergence-rate results include~\cite{Talbi2024FiniteDimensional,CecchinDaudinJacksonMartini2024,BayraktarEkrenZhou2025CBO}. The comparison framework developed here suggests how such questions may be revisited for state- and law-dependent common-noise directions.

Let $B,W$ be two independent Brownian motions on a filtered probability space $(\Omega, \mathcal{F},\mathbb P)$, and let $A \subset \R$ be a compact control set. Consider coefficients $b, \sigma:\R \x \Pc(\R)\times A \to \R$, $\sigma_0:\R \x \Pc_2(\R) \to \R$.
We consider the McKean--Vlasov SDE
\begin{align*}
dX^{t,\mu,\alpha}_s&=b(X^{t,\mu,\alpha}_s,m^{t,\mu,\alpha}_s,\alpha_s) \, ds + \sigma(X^{t,\mu,\alpha}_s,m^{t,\mu,\alpha}_s,\alpha_s) \, dB_s + \sigma_0(X^{t,\mu,\alpha}_s, m^{t,\mu,\alpha}_s) \, dW_s, \, \text{$t \leq s \leq T$}, \\
X_t^{t,\mu,\alpha}&=\xi,
~
m^{t,\mu,\alpha}_s := \Lc(X^{t,\mu,\alpha}_s \, |\,\Fc^W_s),
\end{align*}
where $\xi$ is independent of $B,W$ with distribution $\mu \in \Pc_2(\R)$ and $\alpha: \Omega \x [0,T] \rightarrow A$ is an admissible control adapted to the filtration $\F$ generated by $W,B,\xi$.

Take $\mathcal{A}:=\{\alpha=(\alpha_s)_{0\leq s \leq T}: \alpha : \Om \x [0,T] \longrightarrow A \text{ is measurable and $\F$-adapted} \}$ to be the set of all admissible controls. Given a running cost $r:\R \times \Pc(\R) \x A \to \R$, and a terminal cost $l: \R \to \R$, we define
\begin{align*}
    J(t,\mu, \alpha):=\E \left[\int_t^T r(X^{t,\mu,\alpha}_s, m^{t,\mu,\alpha}_s, \alpha_s) \, ds + l(X^{t,\mu,\alpha}_T)  \right].
\end{align*}
The optimization problem reads
\begin{align*}
    v(t,\mu)=\inf_{\alpha \in \mathcal{A}} J(t,\mu,\alpha).
\end{align*}
We define, for $(a,\mu,p,q,M) \in A  \times \Pc_2(\R) \times B_{q} \times B_{q} \times \R $,
\begin{align*}
    \bar K(a,\mu,p,q,M):=& \int  r(x,\mu,a)+ b(x,\mu,a) p(x) +\frac{1}{2} q(x) \sigma^2 (x,\mu,a) \, \mu(dx) +\frac{1}{2} M.
\end{align*}

\begin{proposition}[Viscosity solution]\label{thm:viscosity_property_common_noise}
    Under Assumptions~\ref{assump:regularity} and~\ref{assumption:application}, the value function is continuous and is the unique viscosity solution of the equation
    \begin{align*}
        -\pa_t v(t,\mu)=& \inf_{a \in A} \bar K(a, \mu, D_{\mu} v(t,\mu), D_{x \mu} v(t,\mu), \mathcal{H}_{\sigma_0}v(t,\mu)) \notag \\
        v(T,\mu)=&\mu(l).
    \end{align*}    
\end{proposition}

\section{Lamperti structure of the transform}\label{sec:lamperti}
We now explain the relationship between the change of variables used above and the classical Lamperti transform. This discussion is not needed for the proof of the comparison theorem, but it clarifies why the auxiliary flow removes the common-noise Hessian from the equation. Throughout this section, $\sigma_0$ satisfies Assumption~\ref{assump:regularity}.

\subsection{Equivalent flow formulations}\label{subsec:equivalent-flow-formulations}
The flow $\psi$ plays the central role in the discussion below. The next proposition identifies two equivalent formulations of this flow.
\begin{prop}[Equivalence]
    Consider the two equations
    \begin{align*}
    \pa_m\psi(x,\mu,m) &= \pa_x\psi(x,\mu,m)\sigma_0(x,\mu) + \int_\R D_\mu \psi(x,\mu,m)(y)\sigma_0(y,\mu)\mu(dy), \quad 
    \psi(x,\mu,0) = x, \\
    \pa_m \psi(x,\mu,m) &= \sigma_0(\psi(x,\mu,m),\mu_m), \quad 
    \mu_m = \psi(\cdot,\mu,m)_\sharp\mu,
    \quad \psi(x,\mu,0) = x.
\end{align*}
Under Assumption~\ref{assump:regularity}, both equations have unique solutions. Moreover, the solution $\psi$ is the same in the two formulations.
\end{prop}
\begin{proof}
    $(\Longleftarrow)$
    Define an auxiliary system
    \begin{align*}
        dX_s = \sigma_0(X_s,\mu_s)ds,
        \quad
        \mu_s = \psi(\cdot,\mu,s)_\sharp\mu
        \quad
        X_0 = x,~ \mu_0 = \mu.
    \end{align*}
    Existence and uniqueness follow from a standard Banach fixed point argument. The flow property of the auxiliary system then implies that
    \begin{align*}
        \psi(x,\mu,s+t) = \psi(\psi(x,\mu,s),\mu_s,t).
    \end{align*}
    Differentiating at $s = 0$ yields the first system.

    $(\Longrightarrow)$
    Consider an auxiliary system
    \begin{align*}
        dX_s = \sigma_0(X_s,\mu_s)ds,
        \quad
        \pa_s\mu_s(\cdot) + \pa_x(\sigma_0(\cdot,\mu_s)\mu_s(\cdot)) ~ = ~ 0,
        \quad
        X_0 = x,~ \mu_0 = \mu.
    \end{align*}
    Then the first system implies that
    \begin{align*}
        \frac{d}{ds}\psi(X_s,\mu_s,m-s) = 0.
    \end{align*}
    Thus, we have $\psi(x,\mu,m) = \psi(X_s,\mu_s,m-s)$ for all $s$.
    Taking $s = m$ gives
    \begin{align*}
        \psi(x,\mu,m) = X_m, 
        \quad
        \pa_m\psi(x,\mu,m) = \sigma_0(X_m,\mu_m) = \sigma_0(\psi(x,\mu,m),\mu_m).
    \end{align*}
    It remains to prove that
    $
        \mu_m = \psi(\cdot,\mu,m)_\sharp\mu
    $
    uniquely solves 
    $
    \pa_s\mu_s(\cdot) + \pa_x(\sigma_0(\cdot,\mu_s)\mu_s(\cdot)) ~ = ~ 0.
    $
    For any $\varphi \in C_c^\infty$,
    \begin{align*}
        \langle \varphi, \mu_m \rangle
        ~ = ~
        \int \varphi(\psi(x,\mu,m))\mu(dx),
    \end{align*}
    Differentiating with respect to $m$, it follows that
    \begin{align*}
        \frac{d}{dm}\langle \varphi, \mu_m \rangle
        ~ &= ~
        \int \varphi'(\psi(x,\mu,m))\sigma_0(\psi(x,\mu,m),\mu_m)\mu(dx)
        ~ = ~
        \int \varphi'(x)\sigma_0(x,\mu_m)\mu_m(dx)
        \\~ &= ~
        -\langle \varphi,\pa_x(\sigma_0(x,\mu_m)\mu_m(x))\rangle.
    \end{align*}
    To complete the proof, we verify the uniqueness of the solution to the Fokker-Planck equation. Let $\mu_s$ be any solution of $\pa_s\mu_s + \pa_x(\sigma_0(\cdot,\mu_s)\mu_s) = 0$ with $\mu_0 = \mu$. Define the auxiliary flow $\Psi$ by $\pa_m\Psi(x,\mu,m) = \sigma_0(\Psi(x,\mu,m),\mu_m)$ with $\Psi(x,\mu,0) = x$, whose existence follows from the first part of the proof. Set $\tilde{\mu}_s := \Psi(\cdot,\mu,s)_\sharp\mu$. A direct computation shows that $\tilde{\mu}_s$ also satisfies $\pa_s\tilde{\mu}_s + \pa_x(\sigma_0(\cdot,\mu_s)\tilde{\mu}_s) = 0$. By uniqueness for this linear PDE (given the Lipschitz velocity field $\sigma_0(\cdot,\mu_s)$), we deduce $\mu_s = \tilde{\mu}_s = \Psi(\cdot,\mu,s)_\sharp\mu$. The uniqueness of $\psi$ then yields $\Psi = \psi$ and consequently $\mu_s = \psi(\cdot,\mu,s)_\sharp\mu$, as desired.
\end{proof}
\subsection{Relation with the classical Lamperti transform}\label{subsec:classical-lamperti}
The function $\psi$ is related to the classical Lamperti transform when $\sigma_0$ is independent of the measure variable, i.e., $\sigma_0(x,\mu) = \sigma_0(x)$ for all $(x,\mu) \in \R \x \Pc_2(\R)$.
In this special case, $\psi$ solves
\begin{align*}
    \pa_m\psi(x,m) = \pa_x\psi(x,m)\sigma_0(x),\quad \psi(x,0) = x.
\end{align*}
If, in addition, $\sigma_0$ is bounded away from zero, then the Lamperti transform associated with $\sigma_0$ is $\Psi : \R \longrightarrow\R$ given by
\begin{align*}
    \Psi(x) := \int_0^x \frac{1}{\sigma_0(y)}dy.
\end{align*}
Then $\psi(x,m) = \Psi^{-1}(\Psi(x)+m)$. 

The present transform has two advantages over the classical Lamperti transform. First, it does not require $\sigma_0$ to be bounded away from zero. Second, it allows the transformation itself to depend on the law. The classical Lamperti transform normalizes a one-dimensional diffusion coefficient; the transform used here removes the common-noise volatility direction from the conditional law dynamics. Thus, in the mean field setting, it reduces dynamics with common noise to transformed dynamics without an explicit common-noise martingale term.

Consider, for example, the mean field control problem with common noise
\begin{align*}
    dX_t = b(t,X_t,\mu_t,\alpha_t)dt + \sigma(t,X_t,\mu_t,\alpha_t)dW_t + \sigma_0(t,X_t,\mu_t)dB_t,
\end{align*}
where $\alpha = (\alpha_t)_{t \in [0,T]}$ is the control, $B,W$ are two independent Brownian motions,
$\mu_t := \Lc(X_t|B)$.
Consider two functions
\[
\psi,\phi:[0,T]\times\mathbb R\times\mathcal P_2(\mathbb R)\times\mathbb R
\longrightarrow \mathbb R,
\]
which are the unique solutions to the following equation
\[
\left\{
\begin{aligned}
\partial_m\psi(t,x,\mu,m)
&=
\partial_x\psi(t,x,\mu,m)\sigma_0(t,x,\mu)
+
\int_{\mathbb R}
D_\mu\psi(t,x,\mu,m)(y)\sigma_0(t,y,\mu)\mu(dy),
\\
\psi(t,x,\mu,0)&=x.
\end{aligned}
\right.
\]
For each fixed \(t,m,\mu\), \(\psi(t,\cdot,\mu,m)\) is a diffeomorphism of
\(\mathbb R\) to \(\mathbb R\). Denote
\[
\Phi_{t,m}(\mu):=\psi(t,\cdot,\mu,m)_{\#}\mu .
\]

Then the dynamics we study is
\[
\widetilde X_t=\psi(t,X_t,\mu_t,-B_t)
\]
instead of \(X\) itself. By Itô's formula on
\(\mathbb R\times\mathcal P_2(\mathbb R)\times\mathbb R\), one has that
\[
\begin{aligned}
d\widetilde X_t
={}&
\Big[
\mathcal L^{\alpha_t}_{t,\mu_t,-B_t}\psi
\Big](X_t)dt
+
\partial_x\psi(t,X_t,\mu_t,-B_t)
\sigma(t,X_t,\mu_t,\alpha_t)dW_t
\\
&+
\Bigg[
\partial_x\psi(t,X_t,\mu_t,-B_t)\sigma_0(t,X_t,\mu_t)
+
\int_{\mathbb R}
D_\mu\psi(t,X_t,\mu_t,-B_t)(y)
\sigma_0(t,y,\mu_t)\mu_t(dy)
\\
&\qquad\qquad
-
\partial_m\psi(t,X_t,\mu_t,-B_t)
\Bigg]dB_t +
\frac12
\Big[
\mathcal R^0_{t,\mu_t,-B_t}\psi
\Big](X_t)dt
\\
={}&
\Big[
\widehat{\mathcal L}^{\alpha_t}_{t,\mu_t,-B_t}\psi
\Big](X_t)dt
+
\partial_x\psi(t,X_t,\mu_t,-B_t)
\sigma(t,X_t,\mu_t,\alpha_t)dW_t .
\end{aligned}
\]
Here the \(dB_t\)-term vanishes by the defining equation of \(\psi\), and
\[
\begin{aligned}
\Big[
\mathcal L^a_{t,\mu,m}\psi
\Big](x)
:={}&
\partial_t\psi(t,x,\mu,m)
+
\partial_x\psi(t,x,\mu,m)b(t,x,\mu,a)
+
\frac12\partial^2_{xx}\psi(t,x,\mu,m)\sigma^2(t,x,\mu,a)
\\
&+
\int_{\mathbb R}
D_\mu\psi(t,x,\mu,m)(y)b(t,y,\mu,a)\mu(dy)
\\
&+
\frac12
\int_{\mathbb R}
\partial_yD_\mu\psi(t,x,\mu,m)(y)
\sigma^2(t,y,\mu,a)\mu(dy).
\end{aligned}
\]
The remaining common-noise Itô correction is
\[
\Big[
\mathcal R^0_{t,\mu,m}\psi
\Big](x)
=
-
\partial_x\psi(t,x,\mu,m)\mathfrak c_0(t,x,\mu)
-
\int_{\mathbb R}
D_\mu\psi(t,x,\mu,m)(y)\mathfrak c_0(t,y,\mu)\mu(dy),
\]
where
\[
\mathfrak c_0(t,x,\mu)
:=
\partial_x\sigma_0(t,x,\mu)\sigma_0(t,x,\mu)
+
\int_{\mathbb R}
D_\mu\sigma_0(t,x,\mu)(y)\sigma_0(t,y,\mu)\mu(dy).
\]
Therefore,
\[
\begin{aligned}
\Big[
\widehat{\mathcal L}^a_{t,\mu,m}\psi
\Big](x)
:={}&
\Big[
\mathcal L^a_{t,\mu,m}\psi
\Big](x)-
\frac12
\partial_x\psi(t,x,\mu,m)\mathfrak c_0(t,x,\mu)
-
\frac12
\int_{\mathbb R}
D_\mu\psi(t,x,\mu,m)(y)\mathfrak c_0(t,y,\mu)\mu(dy).
\end{aligned}
\]

Moreover,
\[
\widetilde\mu_t:=\mathcal L(\widetilde X_t| B)
=
\Phi_{t,-B_t}(\mu_t),
\qquad
\mu_t=
\Phi_{t,B_t}(\widetilde\mu_t).
\]
In other words, the dynamics of \(\widetilde X\) has no common noise volatility term and follows
\[
d\widetilde X_t
=
\widetilde b(t,\widetilde X_t,\widetilde\mu_t,B_t,\alpha_t)dt
+
\widetilde\sigma(t,\widetilde X_t,\widetilde\mu_t,B_t,\alpha_t)dW_t,
\]
where \(\widetilde\mu_t:=\mathcal L(\widetilde X_t| B)\), and
\[
\begin{aligned}
\widetilde b(t,x,\mu,m,a)
&:=
\Big[
\widehat{\mathcal L}^{a}_{t,\Phi_{t,m}(\mu),-m}\psi
\Big](\psi(t,x,\mu,m)),
\\
\widetilde\sigma(t,x,\mu,m,a)
&:=
\partial_x\psi(t,\psi(t,x,\mu,m),\Phi_{t,m}(\mu),-m)
\sigma(t,x^{m,\mu},\Phi_{t,m}(\mu),a).
\end{aligned}
\]

Note that \(\widetilde\mu,\mu,X,\widetilde X\) have the following relationship
\[
\widetilde\mu_t=\Phi_{t,-B_t}(\mu_t),
\qquad
\mu_t=\Phi_{t,B_t}(\widetilde\mu_t),
\]
\[
\widetilde X_t=\psi(t,X_t,\mu_t,-B_t),
\qquad
X_t=\psi(t,\widetilde X_t,\widetilde\mu_t,B_t).
\]

\section{Finite-measure equations from nonlinear filtering}\label{sec:finite-measure}
The same idea also applies to PDEs on spaces of finite positive measures. This is important for filtering, because the unnormalized conditional law in the Zakai equation is naturally finite-measure valued, while the normalized conditional law in the Kushner--Stratonovich equation is probability-measure valued. Martini~\cite{Martini2023,Martini2024} studies precisely these Kolmogorov equations on spaces of measures associated with nonlinear filtering. The transformations below show how such equations fit into the same change-of-variable framework used in the preceding sections. In particular, they indicate how the Crandall--Ishii viscosity comparison method developed here can be extended beyond probability measures to the finite-measure PDEs generated by filtering, a class that appears to be outside the scope of the existing Wasserstein comparison literature.
\subsection{Equations with linear functional derivatives}\label{subsec:finite-linear-functional}
Let $\Mc_2(\R^d)$ denote the space of finite positive measures with finite second moment, and consider the following second-order PDE motivated by the Zakai equation:
\begin{align}\label{eq:Zakai_simple}
    -\left(\pa_t u+G(\cdot, \delta_\mu u,\Hc_1 u)\right)(t,\mu)=0,
    ~(t,\mu) \in [0,T) \x \Mc_2(\R),
\end{align}
where for any $v : \Mc_2(\R^d) \longrightarrow \R$ regular enough,
\begin{align*}
    \Hc_1 v(\mu) := \int_{\R^d \x \R^d} \delta^2_{\mu\mu}v(\mu,x,y)h^T(x)h(y)\mu(dx)\mu(dy).
\end{align*}
Setting $\Lc^m\mu := e^{mh(\cdot)}\mu$, one can transform \eqref{eq:Zakai_simple} into an augmented PDE and prove a comparison principle for standard viscosity solutions under suitable structural conditions.

\subsection{Equations with Lions and linear functional derivatives}\label{subsec:finite-lions-linear}
Let $\Mc_2(\R^d)$ denote the space of finite positive measures with finite second moment, and consider the following second-order PDE motivated by the Zakai equation:
\begin{align}\label{eq:Zakai}
    -\left(\pa_t u+G(\cdot, \delta_\mu u, D_\mu u,D_{x\mu}u,\Hc_2 u)\right)(t,\mu)=0,
    ~(t,\mu) \in [0,T) \x \Mc_2(\R),
\end{align}
where for any $v : \Mc_2(\R) \longrightarrow \R$ regular enough,
\begin{align*}
    \Hc_2 v(\mu) := 
    \begin{bmatrix}
        \int_{\R \x \R} \delta^2_{\mu\mu}v(\mu,x,y)\mu(dx)\mu(dy) 
        &
        \int_{\R \x \R}\sigma^\top_0(x)\pa_x\delta^2_{\mu\mu}u(\mu,x,y)\mu(dx)\mu(dy) 
        \\
        \int_{\R \x \R}\pa_y\delta^2_{\mu\mu}u(\mu,x,y)\sigma_0(y)\mu(dx)\mu(dy) 
        &
        \Hc u 
    \end{bmatrix}.
\end{align*}
Setting $\Lc^{m,n}\mu := e^{n}(\psi(\cdot,m)_\sharp\mu)$, where $\psi$ solves
\begin{align*}
    \pa_m\psi(x,m) = \pa_x\psi(x,m)\sigma_0(x), \quad \psi(x,0) = x,
\end{align*}
one obtains an augmented PDE for which the same comparison strategy can be applied. In particular, this transformation covers the mixed finite-measure and probability-measure structures that appear in filtering equations of Zakai and Kushner--Stratonovich type.

\acknowledgement{We thank Jianfeng Zhang at USC for fruitful discussions.}

\appendix

\section{A Sobolev commutator estimate}\label{sec:commutator}

The class of Schwartz functions  $\mathcal{S}(\mathbb R)$ is the space of smooth functions whose derivatives are bounded by $C_N(1+|\xi|^2)^{-N}$ for every $N \in \mathbb Z^+$, equipped with the topology induced by the family of seminorms
\begin{align*}
    \rho_{\alpha,\beta}(\phi)=\sup_{\xi \in \mathbb R}  \left|\xi^{\alpha} \partial_{\beta}\phi(\xi) \right|, \quad \forall \, \phi \in \mathcal{S}({\mathbb R}),
\end{align*}
indexed by all multi-indices $\alpha,\beta$. Denote by $\mathcal{S}'(\mathbb R)$ the topological dual of $\mathcal{S}(\mathbb R)$.

\begin{definition}\label{def:Sobolev}
    Let $s$ be a real number. The space $H^s(\mathbb R)$ is defined as the set of all distributions $u$ in $\mathcal{S}'(\mathbb R)$ with the property that
  \begin{align}
        \mathcal{J}_{-s} (u):= (I_d - \Delta)^{s/2} u = \mathcal{F}^{-1} ((1+|\xi|^2)^{s/2} \mathcal{F}u(\cdot) )
    \end{align}
    is an element of $L^2(\mathbb R)$. Here, $\mathcal{F}$ denotes the Fourier transform. $\Jc_{-s}$ is called the Bessel potential operator, and $|u|_s:=|\Jc_{-s} u |_{L^2}=|(1+|\xi|^2)^{s/2} \mathcal{F}u(\cdot)|_{L^2} $.
\end{definition}
These are the standard Bessel potential conventions; see, for example,~\cite{AdamsHedberg1996,Grafakos2014}. The commutator estimate below is the same type of Sobolev estimate that appears in the HJB theory for controlled Duncan--Mortensen--Zakai equations~\cite{GozziSwiech2000}.

For given functions $a \in W^{\lceil\lambda + 3/2\rceil,\infty}(\R),c \in W^{\lceil\lambda\rceil,\infty}(\R)$, define the operators
\begin{align*}
\mathcal{B}f(x)&=c(x)D_xf(x),\,\mathcal{A}f(x)=\frac{1}{2}a(x)D^2_x f(x).\\
\end{align*}

\begin{proposition}\label{prop:commutator}
Assume that $\xi^{\top} a(x)\xi \geq \delta |\xi|^2$ for all $x,\xi \in \R$.
    For any finite signed measure $\eta$, we have
    \begin{align}\label{eq:boundf}
        \int_{\R} \left( \Ac(\Jc_{2\lambda}\eta)(x) + \Bc(\Jc_{2\lambda}\eta)(x)\right)\eta(dx) \leq -\frac{\d}{4} |\eta|_{1-\lambda}^2+c|\eta|_{-\lambda}^2
    \end{align}
    for a strictly positive constant $c$.
\end{proposition}

\section{Flow estimates for the nonlinear transform}\label{sec:flow-estimates}
This section records the key estimates related to the change-of-variable functions $\psi$ and $\phi$.

\begin{theorem}\label{thm:spatial-regularity}
Suppose Assumption~\ref{assump:regularity} holds. Then for each $m \in [0,T]$ and each probability measure $\mu$, the backward flow $\phi(\cdot, \mu, m)$ and its Lions derivative $D_\mu \phi$ satisfy
\begin{equation}\label{eq:spatial-bound}
\begin{split}
\sup_{\mu \in \Pc_2(\R),y \in \R}\bigg(\sum_{i = 1}^{k+1}\norm{\partial_x^i \phi(\cdot, \mu,m)}_{L^\infty} 
+ \sum_{i = 1}^{k}\Big(\norm{\pa_x^iD_{y\mu} \phi(\cdot, \mu, m)(y)}_{L^\infty} +\norm{\pa_x^iD_{\mu} \phi(\cdot, \mu, m)(y)}_{L^\infty}\Big)\bigg)
\\ \leq C\,e^{C|m|}.
\end{split}
\end{equation}
The same result holds for the forward flow $\psi$.
\end{theorem}

\begin{proof}[Proof of Theorem \ref{thm:spatial-regularity}]
We prove the estimate for the backward flow \(\phi\). The proof for the
forward flow \(\psi\) is identical, with the sign of the vector field
changed.

Throughout the proof, \(C\) denotes a constant depending only on the
constants in Assumption \(2.1\), on \(k\), and on \(T\). Its value may
change from line to line.

\textbf{First we estimate the spatial derivatives.} Differentiating
\[
        \partial_m\phi(x,\mu,m)
        =
        -\sigma_0(\phi(x,\mu,m),\tilde\mu_m)
\]
with respect to \(x\), and observing that \(\tilde\mu_m\) is fixed
when differentiating in the spatial variable \(x\), gives
\[
        \partial_m\partial_x\phi(x,\mu,m)
        =
        -\partial_x\sigma_0(\phi(x,\mu,m),\tilde\mu_m)
        \partial_x\phi(x,\mu,m),
        \qquad
        \partial_x\phi(x,\mu,0)=1.
\]
Hence
\[
        \partial_x\phi(x,\mu,m)
        =
        \exp\left(
        -\int_0^m
        \partial_x\sigma_0(\phi(x,\mu,\ell),\tilde\mu_\ell)
        \,d\ell
        \right).
\]
Since \(\|\partial_x\sigma_0\|_\infty\le C\), we obtain
        $e^{-Cm}
        \le
        \partial_x\phi(x,\mu,m)
        \le
        e^{Cm}.$
In particular, \(\phi(\cdot,\mu,m)\) is a \(C^1\)-diffeomorphism of
\(\mathbb R\).

Let \(2\le n\le k+1\). Differentiating the equation \(n\) times with
respect to \(x\), chain rule gives
\[
        \partial_m\partial_x^n\phi(x,\mu,m)
        =
        -\partial_x\sigma_0(\phi(x,\mu,m),\tilde\mu_m)
        \partial_x^n\phi(x,\mu,m)
        +
        R_{n,m}(x),
\]
where \(R_{n,m}\) is a finite sum of terms of the form
\[
        C_{\alpha}
        \partial_x^\ell\sigma_0(\phi(x,\mu,m),\tilde\mu_m)
        \prod_{j=1}^{n-1}
        \bigl(\partial_x^j\phi(x,\mu,m)\bigr)^{\alpha_j},
\]
with \(2\le \ell\le n\) and
        $\sum_{j=1}^{n-1}j\alpha_j=n.$
By Assumption \(2.1\),
        $\sup_{\nu\in\mathcal P_2(\mathbb R)}
        \|\partial_x^\ell\sigma_0(\cdot,\nu)\|_\infty
        \le C,
        0\le \ell\le k+2.$
Assume by induction that, for every \(1\le j<n\),
        $\|\partial_x^j\phi(\cdot,\mu,m)\|_\infty
        \le C e^{Cm}.$
        
Then we have
        $\|R_{n,m}\|_\infty
        \le C e^{Cm},$
and
\[
        \frac{d}{dm}
        \|\partial_x^n\phi(\cdot,\mu,m)\|_\infty
        \le
        C\|\partial_x^n\phi(\cdot,\mu,m)\|_\infty
        +
        C e^{Cm}.
\]
Since \(\partial_x^n\phi(x,\mu,0)=0\) for \(n\ge2\), Gronwall's
inequality yields
\[
        \|\partial_x^n\phi(\cdot,\mu,m)\|_\infty
        \le C e^{Cm},
        \qquad 2\le n\le k+1.
\]
Combining this estimate with the bound for the first derivative gives
\[
        \sum_{i=1}^{k+1}
        \|\partial_x^i\phi(\cdot,\mu,m)\|_\infty
        \le C e^{Cm}.
\]

\textbf{We now estimate the Lions derivative.} Set
        $U_m(x,y):=D_\mu\phi(x,\mu,m)(y)$, we obtain
\[
\begin{aligned}
        \partial_m U_m(x,y)
        ={}&
        -\partial_x\sigma_0(\phi(x,\mu,m),\tilde\mu_m)
        U_m(x,y)                                      \\
        &-
        D_\mu\sigma_0(\phi(x,\mu,m),\tilde\mu_m)
        (\phi(y,\mu,m))
        \partial_x\phi(y,\mu,m)                       \\
        &-
        \int_{\mathbb R}
        D_\mu\sigma_0(\phi(x,\mu,m),\tilde\mu_m)
        (\phi(z,\mu,m))
        U_m(z,y)\,\mu(dz),
\end{aligned}
\]
with \(U_0(x,y)=0\).
Using the boundedness of \(D_\mu\sigma_0\), the spatial derivative
estimate for \(\phi\), and the preceding equation, we get
\[
        |U_m(x,y)|
        \le
        C\int_0^m
        \left(
        |U_\ell(x,y)|
        +
        e^{C\ell}
        +
        \int_{\mathbb R}|U_\ell(z,y)|\,\mu(dz)
        \right)d\ell.
\]
Taking the supremum over \(x,y\in\mathbb R\), we obtain
\[
        \sup_{x,y\in\mathbb R}|U_m(x,y)|
        \le
        C\int_0^m
        \left(
        \sup_{x,y\in\mathbb R}|U_\ell(x,y)|
        +
        e^{C\ell}
        \right)d\ell.
\]
By Gronwall's inequality,
\[
        \sup_{x,y\in\mathbb R}|U_m(x,y)|
        \le
        C e^{Cm}.
\]
For higher orders, and $D_{y\mu}\phi(\cdot,\mu,m)(y)$, the steps are similar.
\end{proof}

\begin{theorem}\label{thm:Lipschitz}
Suppose Assumption~\ref{assump:regularity} holds. 
For every \(M>0\) there exists \(C_M>0\) such that, for all \(|m|,|m'|\le M\), all \(x,x',y,y'\in\R\), and all \(\mu,\nu\in\mathcal P_2(\R)\),
\[
\begin{aligned}
&\abs{\phi(x,\mu,m)-\phi(x',\nu,m')}
+
\abs{\partial_x\phi(x,\mu,m)-\partial_x\phi(x',\nu,m')}
\\
&\quad+
\abs{D_\mu\phi(x,\mu,m)(y)-D_\mu\phi(x',\nu,m')(y')}
\\
&\quad+
\abs{D_{y\mu}\phi(x,\mu,m)(y)-D_{y\mu}\phi(x',\nu,m')(y')}
\\
&\le
Ce^{CM}\left(
\abs{x-x'}+\abs{y-y'}+\abs{m-m'}+|\mu-\nu|_{-\lambda}
\right).
\end{aligned}
\]
The same estimate holds with \(\phi\) replaced by \(\psi\).

Moreover, for every \(0\le s\le k+1\) there exists \(C_s>0\) such that, for every \(f\in H^s(\R)\), every \(\mu\in\mathcal P_2(\R)\), and every \(m\in\R\),
\[
 \abs{f(\psi(\cdot,\mu,m))}_{s}
 \le
 C_s e^{C_s\abs m}\abs{f}_{s}.
\]
\end{theorem}
\begin{proof}[Proof of Theorem \ref{thm:Lipschitz}]
Throughout the proof, the constant $C$ may change from line to line, but it depends only on the constants appearing in Assumption~\ref{assump:regularity} and on the fixed time horizon under consideration. If the range of the variable $m$ is restricted by $|m|\leq M$, the corresponding constants are denoted by $C_M$ and may be chosen of the form $Ce^{CM}$ for some $C$.

\medskip

\noindent \textbf{We first prove the estimates for $\phi$.} Recall that $\phi$ is defined as the solution of the ordinary differential equation
\[
\partial_m \phi(x,\mu,m)
=
\sigma_0(\phi(x,\mu,m),\mu),
\qquad
\phi(x,\mu,0)=x .
\]
Since $\sigma_0$ is bounded and globally Lipschitz in the variables $(x,\mu)$, the equation has a unique global solution for every $(x,\mu)\in \mathbb{R}\times \mathcal{P}_2(\mathbb{R})$.

Let $x,x'\in\mathbb{R}$, $\mu,\nu\in \mathcal{P}_2(\mathbb{R})$, and $m,m'\in\mathbb{R}$. For the spatial and measure variables, we write
\[
\begin{aligned}
|\phi(x,\mu,m)-\phi(x',\nu,m)|
&\leq |x-x'|
   +\int_0^{|m|}
     |\sigma_0(\phi(x,\mu,r),\mu)
       -\sigma_0(\phi(x',\nu,r),\nu)|\,dr  \\
&\leq |x-x'|
   +L\int_0^{|m|}
      |\phi(x,\mu,r)-\phi(x',\nu,r)|\,dr
   +L|m|\,|\mu-\nu|_{-\lambda}.
\end{aligned}
\]
By Gronwall's lemma,
\[
|\phi(x,\mu,m)-\phi(x',\nu,m)|
\leq C_M\bigl(|x-x'|+|\mu-\nu|_{-\lambda}\bigr),
\qquad |m|\leq M .
\]
Moreover, using the boundedness of $\sigma_0$,
\[
|\phi(x',\nu,m)-\phi(x',\nu,m')|
\leq
\|\sigma_0\|_{L^\infty}|m-m'|.
\]
Combining the two estimates yields
\[
|\phi(x,\mu,m)-\phi(x',\nu,m')|
\leq
C_M\bigl(
|x-x'|+|m-m'|+|\mu-\nu|_{-\lambda}
\bigr).
\]

We now estimate the spatial derivative. Differentiating the equation for $\phi$ with respect to $x$ gives
\[
\partial_m \partial_x\phi(x,\mu,m)
=
\partial_x\sigma_0(\phi(x,\mu,m),\mu)\,
\partial_x\phi(x,\mu,m),
\qquad
\partial_x\phi(x,\mu,0)=1 .
\]
Hence
\[
\partial_x\phi(x,\mu,m)
=
\exp\left(
\int_0^m
\partial_x\sigma_0(\phi(x,\mu,r),\mu)\,dr
\right).
\]
Since $\partial_x\sigma_0$ is bounded, there exists $C_M>0$ such that
\[
C_M^{-1}
\leq
\partial_x\phi(x,\mu,m)
\leq
C_M,
\qquad |m|\leq M .
\]
In particular, $x\mapsto \phi(x,\mu,m)$ is a $C^1$-diffeomorphism of $\mathbb{R}$.

We next prove the Lipschitz estimate for $\partial_x\phi$. From the integral equation for $\partial_x\phi$, we have
\[
\begin{aligned}
&|\partial_x\phi(x,\mu,m)-\partial_x\phi(x',\nu,m)|   \\
&\quad \leq
\int_0^{|m|}
\left|
\partial_x\sigma_0(\phi(x,\mu,r),\mu)
-
\partial_x\sigma_0(\phi(x',\nu,r),\nu)
\right|
|\partial_x\phi(x,\mu,r)|\,dr  \\
&\qquad
+
\int_0^{|m|}
|\partial_x\sigma_0(\phi(x',\nu,r),\nu)|
|\partial_x\phi(x,\mu,r)-\partial_x\phi(x',\nu,r)|\,dr .
\end{aligned}
\]
Using Assumption 2.1 and the already established estimate for $\phi$, we get
\[
|\partial_x\phi(x,\mu,m)-\partial_x\phi(x',\nu,m)|
\leq
C_M\bigl(|x-x'|+|\mu-\nu|_{-\lambda}\bigr).
\]
The dependence on $m$ follows from
\[
|\partial_x\phi(x',\nu,m)-\partial_x\phi(x',\nu,m')|
\leq
C_M |m-m'|,
\]
which is obtained directly from the differential equation for $\partial_x\phi$. Therefore
\[
|\partial_x\phi(x,\mu,m)-\partial_x\phi(x',\nu,m')|
\leq
C_M\bigl(
|x-x'|+|m-m'|+|\mu-\nu|_{-\lambda}
\bigr).
\]

We next consider the Lions derivative. Let $D_\mu\phi(x,\mu,m)(y)$ denote the derivative of $\phi$ with respect to the measure variable. Differentiating the flow equation in the Lions sense gives
\[
\begin{aligned}
\partial_m D_\mu\phi(x,\mu,m)(y)
=
\partial_x\sigma_0(\phi(x,\mu,m),\mu)
D_\mu\phi(x,\mu,m)(y) 
+
D_\mu\sigma_0(\phi(x,\mu,m),\mu)(y),
\end{aligned}
\]
with initial condition $D_\mu\phi(x,\mu,0)(y)=0.$
By variation of constants,
\[
\begin{aligned}
D_\mu\phi(x,\mu,m)(y)
=
\int_0^m
\exp\left(
\int_r^m
\partial_x\sigma_0(\phi(x,\mu,\ell),\mu)\,d\ell
\right) 
D_\mu\sigma_0(\phi(x,\mu,r),\mu)(y)\,dr .
\end{aligned}
\]
Since $D_\mu\sigma_0$ and $\partial_x\sigma_0$ are bounded, we obtain
\[
|D_\mu\phi(x,\mu,m)(y)|
\leq C_M .
\]

We next prove the Lipschitz estimate for $D_\mu\phi$. For simplicity, we write
\[
U_m(x,\mu,y)=D_\mu\phi(x,\mu,m)(y).
\]
Using the equation satisfied by $U_m$, the Lipschitz assumptions on $\partial_x\sigma_0$ and $D_\mu\sigma_0$, and the previously obtained estimate for $\phi$, we get
\[
\begin{aligned}
&|U_m(x,\mu,y)-U_m(x',\nu,y')|  \\
&\quad \leq
C_M\int_0^{|m|}
\bigl(
|\phi(x,\mu,r)-\phi(x',\nu,r)|
+|y-y'|
+|\mu-\nu|_{-\lambda}
\bigr)\,dr  \\
&\qquad
+
C_M\int_0^{|m|}
|U_r(x,\mu,y)-U_r(x',\nu,y')|\,dr .
\end{aligned}
\]
By Gronwall's lemma,
\[
|D_\mu\phi(x,\mu,m)(y)-D_\mu\phi(x',\nu,m)(y')|
\leq
C_M\bigl(
|x-x'|+|y-y'|+|\mu-\nu|_{-\lambda}
\bigr).
\]
The dependence on $m$ is again obtained from the differential equation for $D_\mu\phi$:
\[
|D_\mu\phi(x',\nu,m)(y')-D_\mu\phi(x',\nu,m')(y')|
\leq C_M |m-m'|.
\]
Thus
\[
\begin{aligned}
&|D_\mu\phi(x,\mu,m)(y)-D_\mu\phi(x',\nu,m')(y')|  \\
&\qquad \leq
C_M\bigl(
|x-x'|+|y-y'|+|m-m'|+|\mu-\nu|_{-\lambda}
\bigr).
\end{aligned}
\]

It remains to estimate $D_{x\mu}\phi$. Differentiating the equation for $D_\mu\phi$ with respect to $x$ gives
\[
\begin{aligned}
\partial_m D_{x\mu}\phi(x,\mu,m)(y)
&=
\partial_{xx}\sigma_0(\phi(x,\mu,m),\mu)
\partial_x\phi(x,\mu,m)
D_\mu\phi(x,\mu,m)(y)   \\
&\quad
+
\partial_x\sigma_0(\phi(x,\mu,m),\mu)
D_{x\mu}\phi(x,\mu,m)(y)  \\
&\quad
+
\partial_xD_\mu\sigma_0(\phi(x,\mu,m),\mu)(y)
\partial_x\phi(x,\mu,m),
\end{aligned}
\]
with
\[
D_{x\mu}\phi(x,\mu,0)(y)=0 .
\]
All coefficients in this linear equation are bounded by Assumption 2.1 and by the bounds already proved for $\phi$, $\partial_x\phi$, and $D_\mu\phi$. Therefore
\[
|D_{x\mu}\phi(x,\mu,m)(y)|
\leq C_M .
\]
Taking two triples $(x,\mu,y)$ and $(x',\nu,y')$, subtracting the corresponding equations, and using the Lipschitz assumptions on
$\partial_{xx}\sigma_0$, $\partial_x\sigma_0$, and $\partial_xD_\mu\sigma_0$, together with the estimates already obtained, gives
\[
\begin{aligned}
&|D_{x\mu}\phi(x,\mu,m)(y)-D_{x\mu}\phi(x',\nu,m)(y')|  \\
&\quad \leq
C_M\bigl(
|x-x'|+|y-y'|+|\mu-\nu|_{-\lambda}
\bigr)
+
C_M\int_0^{|m|}
|D_{x\mu}\phi(x,\mu,r)(y)-D_{x\mu}\phi(x',\nu,r)(y')|\,dr .
\end{aligned}
\]
Gronwall's lemma yields
\[
|D_{x\mu}\phi(x,\mu,m)(y)-D_{x\mu}\phi(x',\nu,m)(y')|
\leq
C_M\bigl(
|x-x'|+|y-y'|+|\mu-\nu|_{-\lambda}
\bigr).
\]
The dependence on $m$ follows directly from the equation for $D_{x\mu}\phi$, and hence
\[
\begin{aligned}
&|D_{x\mu}\phi(x,\mu,m)(y)-D_{x\mu}\phi(x',\nu,m')(y')|  \\
&\qquad \leq
C_M\bigl(
|x-x'|+|y-y'|+|m-m'|+|\mu-\nu|_{-\lambda}
\bigr).
\end{aligned}
\]
This proves the desired estimate for $\phi$.

\medskip

\noindent \textbf{It remains to prove the Sobolev composition estimate.} Fix $s\in\{0,\ldots,k+1\}$, $\mu\in\mathcal{P}_2(\mathbb{R})$, and $m\in\mathbb{R}$. We must show that the operator
\[
f\longmapsto f(\psi(\cdot,\mu,m))
\]
is bounded on $H^s(\mathbb{R})$, with norm at most $C_s e^{C_s|m|}$.

For $s=0$, using the change of variables $z=\psi(x,\mu,m)$, or equivalently $x=\phi(z,\mu,m)$, we obtain
\[
\begin{aligned}
\|f(\psi(\cdot,\mu,m))\|_{L^2}^2
&=
\int_{\mathbb{R}} |f(\psi(x,\mu,m))|^2\,dx  \\
&=
\int_{\mathbb{R}} |f(z)|^2
\partial_x\phi(z,\mu,m)\,dz .
\end{aligned}
\]
Therefore
\[
\|f(\psi(\cdot,\mu,m))\|_{L^2}
\leq
e^{C|m|}\|f\|_{L^2}.
\]

Let now $1\leq s\leq k+1$. By chain rule, for every integer $j\in\{1,\ldots,s\}$,
\[
\partial_x^j\bigl(f(\psi(x,\mu,m))\bigr)
=
\sum_{\ell=1}^j
f^{(\ell)}(\psi(x,\mu,m))
P_{j,\ell}
\bigl(
\partial_x\psi(x,\mu,m),
\ldots,
\partial_x^{j-\ell+1}\psi(x,\mu,m)
\bigr),
\]
where each $P_{j,\ell}$ is a polynomial depending only on $j$ and $\ell$.

By Assumption 2.1, the coefficients of the flow are uniformly bounded in $W^{k+1,\infty}$; hence the standard ODE estimates for derivatives of flows give, for every $1\leq r\leq k+1$,
\[
\|\partial_x^r\phi(\cdot,\mu,m)\|_{L^\infty}
+
\|\partial_x^r\psi(\cdot,\mu,m)\|_{L^\infty}
\leq
C_r e^{C_r|m|}.
\]
Consequently,
\[
\left\|
P_{j,\ell}
\bigl(
\partial_x\psi,
\ldots,
\partial_x^{j-\ell+1}\psi
\bigr)
\right\|_{L^\infty}
\leq
C_s e^{C_s|m|}.
\]
It follows that
\[
\|\partial_x^j(f\circ\psi)\|_{L^2}
\leq
C_s e^{C_s|m|}
\sum_{\ell=1}^j
\|f^{(\ell)}\circ\psi\|_{L^2}.
\]
Using again the change of variables $x=\phi(z,\mu,m)$ gives
\[
\|f^{(\ell)}\circ\psi\|_{L^2}
\leq
e^{C|m|}
\|f^{(\ell)}\|_{L^2}.
\]
Therefore,
\[
\|\partial_x^j(f\circ\psi)\|_{L^2}
\leq
C_s e^{C_s|m|}
\sum_{\ell=1}^j
\|f^{(\ell)}\|_{L^2}.
\]
Summing over $j=0,\ldots,s$ yields
\[
\|f(\psi(\cdot,\mu,m))\|_{H^s}
\leq
C_s e^{C_s|m|}\|f\|_{H^s}.
\]
Then for any $s \in [0,k+1]$, the conclusion is also true.
This completes the proof.
\end{proof}

\bibliography{paper_references}

@unpublished{Lions,
  author = {Lions, P.-L.},
  title = {Lectures at Coll{\`e}ge de France},
  note = {www.college-de-france.fr},
  year = {2011}
}

@book{carmona_probabilistic_2018_vol1,
  author = {Carmona, Ren{\'e} and Delarue, Fran{\c c}ois},
  title = {Probabilistic Theory of Mean Field Games with Applications I: Mean Field FBSDEs, Control, and Games},
  series = {Probability Theory and Stochastic Modelling},
  publisher = {Springer International Publishing},
  year = {2018}
}

@book{carmona_probabilistic_2018_vol2,
  author = {Carmona, Ren{\'e} and Delarue, Fran{\c c}ois},
  title = {Probabilistic Theory of Mean Field Games with Applications II: Mean Field Games with Common Noise and Master Equations},
  series = {Probability Theory and Stochastic Modelling},
  publisher = {Springer International Publishing},
  year = {2018}
}

@book{lions_annals,
  author = {Cardaliaguet, Pierre and Delarue, Fran{\c c}ois and Lasry, Jean-Michel and Lions, Pierre-Louis},
  title = {The Master Equation and the Convergence Problem in Mean Field Games},
  series = {Annals of Mathematics Studies},
  volume = {201},
  publisher = {Princeton University Press},
  year = {2019}
}

@book{bensoussan_mean_2013,
  author = {Bensoussan, Alain and Frehse, Jens and Yam, Phillip},
  title = {Mean Field Games and Mean Field Type Control Theory},
  volume = {101},
  publisher = {Springer},
  year = {2013}
}

@article{MR3630288,
  author = {Buckdahn, Rainer and Li, Juan and Peng, Shige and Rainer, Catherine},
  title = {{Mean-field stochastic differential equations and associated PDEs}},
  journal = {Ann. Probab.},
  volume = {45},
  number = {2},
  pages = {824--878},
  year = {2017},
  doi = {10.1214/15-AOP1076}
}

@article{pham_bellman_2018,
  author = {Pham, Huy{\^e}n and Wei, Xiaoli},
  title = {{Bellman equation and viscosity solutions for mean-field stochastic control problem}},
  journal = {ESAIM Control Optim. Calc. Var.},
  volume = {24},
  number = {1},
  pages = {437--461},
  year = {2018},
  doi = {10.1051/cocv/2017019}
}

@article{cosso2021master,
  author = {Cosso, Andrea and Gozzi, Fausto and Kharroubi, Idris and Pham, Huy{\^e}n and Rosestolato, Mauro},
  title = {{Master Bellman equation in the Wasserstein space: uniqueness of viscosity solutions}},
  journal = {Trans. Amer. Math. Soc.},
  volume = {377},
  number = {1},
  pages = {31--83},
  year = {2024},
  doi = {10.1090/tran/8986}
}

@article{BurzoniIgnazioReppenSoner2020,
  author = {Burzoni, Matteo and Ignazio, Vincenzo and Reppen, A. Max and Soner, H. Mete},
  title = {{Viscosity Solutions for Controlled McKean--Vlasov Jump-Diffusions}},
  journal = {SIAM J. Control Optim.},
  volume = {58},
  number = {3},
  pages = {1676--1699},
  year = {2020},
  doi = {10.1137/19M1290061}
}

@article{GangboMeszaros2022,
  author = {Gangbo, Wilfrid and M{\'e}sz{\'a}ros, Alp{\'a}r R.},
  title = {{Global Well-Posedness of Master Equations for Deterministic Displacement Convex Potential Mean Field Games}},
  journal = {Comm. Pure Appl. Math.},
  volume = {75},
  number = {12},
  pages = {2685--2801},
  year = {2022},
  doi = {10.1002/cpa.22069}
}

@article{GangboMeszarosMouZhang2022,
  author = {Gangbo, Wilfrid and M{\'e}sz{\'a}ros, Alp{\'a}r R. and Mou, Chenchen and Zhang, Jianfeng},
  title = {{Mean Field Games Master Equations with Nonseparable Hamiltonians and Displacement Monotonicity}},
  journal = {Ann. Probab.},
  volume = {50},
  number = {6},
  pages = {2178--2217},
  year = {2022},
  doi = {10.1214/22-AOP1580}
}

@article{GangboMayorgaSwiech2021,
  author = {Gangbo, Wilfrid and Mayorga, Sergio and {\'S}wi{\k e}ch, Andrzej},
  title = {{Finite Dimensional Approximations of Hamilton--Jacobi--Bellman Equations in Spaces of Probability Measures}},
  journal = {SIAM J. Math. Anal.},
  volume = {53},
  number = {2},
  pages = {1320--1356},
  year = {2021},
  doi = {10.1137/20M1331135}
}

@article{Talbi2024FiniteDimensional,
  author = {Talbi, Mehdi},
  title = {{A Finite-Dimensional Approximation for Partial Differential Equations on Wasserstein Space}},
  journal = {Stochastic Process. Appl.},
  volume = {177},
  pages = {104445},
  year = {2024},
  doi = {10.1016/j.spa.2024.104445}
}

@article{usersguide,
  author = {Crandall, Michael G. and Ishii, Hitoshi and Lions, Pierre-Louis},
  title = {{User's guide to viscosity solutions of second order partial differential equations}},
  journal = {Bull. Amer. Math. Soc. (N.S.)},
  volume = {27},
  number = {1},
  pages = {1--67},
  year = {1992},
  doi = {10.1090/S0273-0979-1992-00266-5}
}

@article{lions1988viscosity1,
  author = {Lions, P.-L.},
  title = {{Viscosity solutions of fully nonlinear second-order equations and optimal stochastic control in infinite dimensions. I. The case of bounded stochastic evolutions}},
  journal = {Acta Math.},
  volume = {161},
  pages = {243--278},
  year = {1988}
}

@incollection{lions1989viscosity2,
  author = {Lions, Pierre-Louis},
  title = {{Viscosity solutions of fully nonlinear second order equations and optimal stochastic control in infinite dimensions. Part II: optimal control of Zakai's equation}},
  booktitle = {Stochastic Partial Differential Equations and Applications II},
  pages = {147--170},
  publisher = {Springer},
  year = {1989}
}

@article{lions1989viscosity3,
  author = {Lions, Pierre-Louis},
  title = {{Viscosity solutions of fully nonlinear second-order equations and optimal stochastic control in infinite dimensions. III. Uniqueness of viscosity solutions for general second-order equations}},
  journal = {J. Funct. Anal.},
  volume = {86},
  number = {1},
  pages = {1--18},
  year = {1989}
}

@article{soner1988hamilton,
  author = {Soner, H. Mete},
  title = {{On the Hamilton--Jacobi--Bellman equations in Banach spaces}},
  journal = {J. Optim. Theory Appl.},
  volume = {57},
  number = {3},
  pages = {429--437},
  year = {1988}
}

@article{SonerYan2024Torus,
  author = {Soner, H. Mete and Yan, Qinxin},
  title = {{Viscosity Solutions for McKean--Vlasov Control on a Torus}},
  journal = {SIAM J. Control Optim.},
  volume = {62},
  number = {2},
  pages = {903--923},
  year = {2024},
  doi = {10.1137/22M1543732}
}

@article{SonerYan2024Eikonal,
  author = {Soner, H. Mete and Yan, Qinxin},
  title = {{Viscosity Solutions of the Eikonal Equation on the Wasserstein Space}},
  journal = {Appl. Math. Optim.},
  volume = {90},
  number = {1},
  pages = {Paper No. 1},
  year = {2024},
  doi = {10.1007/s00245-024-10145-2}
}

@article{BaEkZh23,
  author = {Bayraktar, Erhan and Ekren, Ibrahim and Zhang, Xin},
  title = {{Comparison of viscosity solutions for a class of second-order PDEs on the Wasserstein space}},
  journal = {Communications in Partial Differential Equations},
  volume = {50},
  number = {4},
  pages = {570--613},
  year = {2025}
}

@article{BayraktarEkrenZhang2025Particle,
  author = {Bayraktar, Erhan and Ekren, Ibrahim and Zhang, Xin},
  title = {{Convergence Rate of Particle System for Second-Order PDEs on Wasserstein Space}},
  journal = {SIAM J. Control Optim.},
  volume = {63},
  number = {3},
  pages = {1768--1782},
  year = {2025},
  doi = {10.1137/24M1684633}
}

@article{BayraktarEkrenZhou2025CBO,
  author = {Bayraktar, Erhan and Ekren, Ibrahim and Zhou, Hongyi},
  title = {{Uniform-in-Time Weak Propagation of Chaos for Consensus-Based Optimization}},
  journal = {Ann. Appl. Probab.},
  year = {2025},
  note = {to appear, arXiv:2502.00582}
}

@article{BayraktarEkrenZhang2023Regret,
  author = {Bayraktar, Erhan and Ekren, Ibrahim and Zhang, Xin},
  title = {{A PDE approach for regret bounds under partial monitoring}},
  journal = {J. Mach. Learn. Res.},
  volume = {24},
  number = {299},
  pages = {1--24},
  year = {2023}
}

@article{BayraktarEkrenZhang2023Variational,
  author = {Bayraktar, Erhan and Ekren, Ibrahim and Zhang, Xin},
  title = {{A Smooth Variational Principle on Wasserstein Space}},
  journal = {Proc. Amer. Math. Soc.},
  volume = {151},
  number = {9},
  pages = {4089--4098},
  year = {2023},
  doi = {10.1090/proc/16466}
}

@article{Bertucci2025StochasticOT,
  author = {Bertucci, Charles},
  title = {{Stochastic Optimal Transport and Hamilton--Jacobi--Bellman Equations on the Set of Probability Measures}},
  journal = {Ann. Inst. H. Poincar{\'e} C Anal. Non Lin{\'e}aire},
  volume = {42},
  number = {6},
  pages = {1543--1600},
  year = {2025},
  doi = {10.4171/AIHPC/138}
}

@article{BertucciLions2024Approx,
  author = {Bertucci, Charles and Lions, Pierre-Louis},
  title = {{An Approximation of the Squared Wasserstein Distance and an Application to Hamilton--Jacobi Equations}},
  journal = {arXiv preprint arXiv:2409.11793},
  year = {2024}
}

@article{CecchinDaudinJacksonMartini2024,
  author = {Cecchin, Alekos and Daudin, Samuel and Jackson, Joe and Martini, Mattia},
  title = {{Quantitative Convergence for Mean Field Control with Common Noise and Degenerate Idiosyncratic Noise}},
  journal = {arXiv preprint arXiv:2409.14053},
  year = {2024}
}

@article{ChowGangbo2019,
  author = {Chow, Yat Tin and Gangbo, Wilfrid},
  title = {{A Partial Laplacian as an Infinitesimal Generator on the Wasserstein Space}},
  journal = {J. Differential Equations},
  volume = {267},
  number = {10},
  pages = {6065--6117},
  year = {2019},
  doi = {10.1016/j.jde.2019.06.012}
}

@article{BEHZ,
  author = {Bayraktar, Erhan and Ekren, Ibrahim and He, Xihao and Zhang, Xin},
  title = {{Comparison for semi-continuous viscosity solutions for second-order PDEs on the Wasserstein space}},
  journal = {J. Differential Equations},
  volume = {455},
  pages = {113963},
  year = {2026},
  doi = {10.1016/j.jde.2025.113963}
}

@article{BCEQTZ25,
  author = {Bayraktar, Erhan and Cheung, Hang and Ekren, Ibrahim and Qiu, Jinniao and Tai, Ho Man and Zhang, Xin},
  title = {{Viscosity Solutions of Fully Second-Order HJB Equations in the Wasserstein Space}},
  journal = {SIAM J. Control Optim.},
  note = {to appear},
  year = {2025}
}

@article{cheung2023viscosity,
  author = {Cheung, Hang and Tai, Ho Man and Qiu, Jinniao},
  title = {{Viscosity Solutions of a Class of Second-Order Hamilton--Jacobi--Bellman Equations in the Wasserstein Space}},
  journal = {Appl. Math. Optim.},
  volume = {91},
  number = {1},
  pages = {23},
  year = {2025}
}

@article{DaJaSe23,
  author = {Daudin, Samuel and Jackson, Joe and Seeger, Benjamin},
  title = {{Well-posedness of Hamilton--Jacobi equations in the Wasserstein space: non-convex Hamiltonians and common noise}},
  journal = {Communications in Partial Differential Equations},
  volume = {50},
  number = {1--2},
  pages = {1--52},
  year = {2025}
}

@article{SaBe24,
  author = {Daudin, Samuel and Seeger, Benjamin},
  title = {{A comparison principle for semilinear Hamilton--Jacobi--Bellman equations in the Wasserstein space}},
  journal = {Calc. Var. Partial Differential Equations},
  volume = {63},
  number = {4},
  year = {2024}
}

@article{zhou2024viscosity,
  author = {Zhou, Jianjun and Touzi, Nizar and Zhang, Jianfeng},
  title = {{Viscosity Solutions for HJB Equations on the Process Space: Application to Mean Field Control with Common Noise}},
  journal = {arXiv preprint arXiv:2401.04920},
  year = {2024}
}

@article{MR3907014,
  author = {Bandini, Elena and Cosso, Andrea and Fuhrman, Marco and Pham, Huy{\^e}n},
  title = {{Randomized filtering and Bellman equation in Wasserstein space for partial observation control problem}},
  journal = {Stochastic Process. Appl.},
  volume = {129},
  number = {2},
  pages = {674--711},
  year = {2019},
  doi = {10.1016/j.spa.2018.03.014}
}

@article{bayraktar2018randomized,
  author = {Bayraktar, Erhan and Cosso, Andrea and Pham, Huy{\^e}n},
  title = {{Randomized dynamic programming principle and Feynman--Kac representation for optimal control of McKean--Vlasov dynamics}},
  journal = {Trans. Amer. Math. Soc.},
  volume = {370},
  number = {3},
  pages = {2115--2160},
  year = {2018}
}

@book{MR2454694,
  author = {Bain, Alan and Crisan, Dan},
  title = {{Fundamentals of Stochastic Filtering}},
  series = {Stochastic Modelling and Applied Probability},
  volume = {60},
  publisher = {Springer, New York},
  year = {2009},
  doi = {10.1007/978-0-387-76896-0}
}

@article{GozziSwiech2000,
  author = {Gozzi, Fausto and {\'S}wi{\k e}ch, Andrzej},
  title = {{Hamilton--Jacobi--Bellman Equations for the Optimal Control of the Duncan--Mortensen--Zakai Equation}},
  journal = {J. Funct. Anal.},
  volume = {172},
  number = {2},
  pages = {466--510},
  year = {2000},
  doi = {10.1006/jfan.2000.3562}
}

@book{AdamsHedberg1996,
  author = {Adams, David R. and Hedberg, Lars Inge},
  title = {{Function Spaces and Potential Theory}},
  series = {Grundlehren der mathematischen Wissenschaften},
  volume = {314},
  publisher = {Springer, Berlin},
  year = {1996}
}

@book{Grafakos2014,
  author = {Grafakos, Loukas},
  title = {{Modern Fourier Analysis}},
  series = {Graduate Texts in Mathematics},
  volume = {250},
  edition = {Third},
  publisher = {Springer, New York},
  year = {2014}
}

@article{behzadan2021multiplication,
  author = {Behzadan, Ali and Holst, Michael},
  title = {{Multiplication in Sobolev spaces, revisited}},
  journal = {Ark. Mat.},
  volume = {59},
  number = {2},
  pages = {275--306},
  year = {2021}
}

@article{Martini2023,
  author = {Martini, Mattia},
  title = {{Kolmogorov equations on spaces of measures associated to nonlinear filtering processes}},
  journal = {Stochastic Process. Appl.},
  volume = {161},
  pages = {385--423},
  year = {2023},
  doi = {10.1016/j.spa.2023.04.013}
}

@article{Martini2024,
  author = {Martini, Mattia},
  title = {{Kolmogorov equations on the space of probability measures associated to the nonlinear filtering equation: the viscosity approach}},
  journal = {Stochastic Anal. Appl.},
  volume = {42},
  number = {6},
  pages = {987--999},
  year = {2024},
  doi = {10.1080/07362994.2024.2408250}
}
\bibliographystyle{plain}
\end{document}